\documentclass[a4paper,12pt]{amsart}
\usepackage{amsmath,amssymb,amsthm}
\usepackage{url}
\usepackage{graphicx}
\usepackage{enumitem}
\usepackage{xcolor,colortbl}
\usepackage{mathtools}
\usepackage[notcite, notref,color, final]{showkeys}
\definecolor{refkey}{rgb}{0,0,1}
\definecolor{labelkey}{rgb}{0,0,1}
\usepackage{soul}
\usepackage[pdftex,hidelinks]{hyperref}
\usepackage{tabularx,booktabs}

\usepackage[margin=1.1in]{geometry}

\newtheorem{lemma}{Lemma}[section]
\newtheorem{theorem}[lemma]{Theorem}
\newtheorem{corollary}[lemma]{Corollary}
\newtheorem{proposition}[lemma]{Proposition}
\newtheorem{algorithm}[lemma]{Algorithm}

\theoremstyle{remark}
\newtheorem{remark}[lemma]{Remark}
\newtheorem{example}[lemma]{Example}

\newcommand{\genericfact}{(a_1\,b_1)\cdots(a_n\,b_n)}
\newcommand{\diff}[1]{\mathrm{diff}(#1)}
\newcommand{\switch}{\gamma}
\newcommand{\map}[3]{{#1}:{#2}\rightarrow{#3}}

\newcommand{\depthGS}{D}
\newcommand{\len}[1]{\ell(#1)}

\newcommand{\supp}[1]{\mathrm{supp}(#1)}

\newcommand{\ninv}{\mathrm{coinv}}
\newcommand{\pinv}{\mathrm{pinv}}
\newcommand{\pninv}{\mathrm{copinv}}
\newcommand{\bninv}{\mathrm{cobinv}}
\newcommand{\bnce}{\mathrm{bounce}}
\newcommand{\inv}{\mathrm{inv}}
\newcommand{\binv}{\mathrm{binv}}
\newcommand{\T}{\mathcal{T}}

\newcommand{\N}{\mathbb{N}}

\newcommand{\F}{\mathcal{F}}

\newcommand{\Sym}[1]{\mathfrak{S}_{#1}}

\newcommand{\M}{\mathcal{M}}

\newcommand{\area}{\mathrm{area}}
\newcommand{\p}{\mathcal{P}}
\newcommand{\tdep}{\mathrm{depth}}
\newcommand{\fact}{f}
\newcommand{\hatfset}{\hat{\F}}
\newcommand{\hatffunc}{\hat{F}}

\newcommand{\jump}[1]{\mathrm{jump}(#1)}
\newcommand{\cojump}[1]{\mathrm{cojump}(#1)}

\newcommand{\attentionpreet}[1]{{\bf \color{red} #1}}
\newcommand{\ignore}[1]{}

\newcommand{\s}{\sigma}
\renewcommand{\a}{\alpha}
\renewcommand{\b}{\beta}
\renewcommand{\t}{\tau}
\newcommand{\pid}{\iota}
\makeatletter
\def\blfootnote{\gdef\@thefnmark{}\@footnotetext}
\makeatother

\newcommand{\cycles}[1]{\mathfrak{C}_{[#1]}}

\newcommand{\AL}{\mathrm{area}_{L}}
\newcommand{\AU}{\mathrm{area}_{U}}

\newcommand{\Section}[1]{Section~\ref{sec:#1}}

\author{John Irving}
\address{Saint Mary's University, Halifax, N.S., Canada}
\email{john.irving@smu.ca}
\author{Amarpreet Rattan}
\address{Department of Mathematics, Simon Fraser University,
Burnaby, BC, Canada}
\email{rattan@sfu.ca}
\title[Trees, parking functions and factorizations of full cycles]
{Trees, parking functions and factorizations of full cycles}

\begin{document}
\blfootnote{Preliminary versions of the main results herein were announced in the extended abstract \cite{irfpsac}.}

\begin{abstract}

Parking functions of length $n$ are well known to be in correspondence
with both labelled trees on $n+1$ vertices and factorizations of the
full cycle $\s_n=(0\,1\,\cdots\,n)$ into $n$ transpositions.  In fact,
these correspondences can be refined: Kreweras equated the area
enumerator of parking functions with the inversion enumerator of  labelled trees, while an elegant bijection of Stanley maps the area of parking functions to a natural statistic on factorizations of $\s_n$.   We extend these relationships in two principal ways. First, we introduce a bivariate refinement of the inversion enumerator of trees and show  that it matches a similarly refined enumerator for factorizations.     Secondly, we characterize all full cycles $\s$ such that Stanley's function remains a bijection when the canonical cycle $\s_n$ is replaced by $\s$.    We also  exhibit a connection between our refined inversion enumerator and  Haglund's bounce statistic on parking functions.
\end{abstract}  

\maketitle
\section{Introduction}
\label{sec:introduction}

This article concerns the interplay between  trees on vertices $\{0,1,\ldots,n\}$,  parking functions of length $n$,  and factorizations of the full cycle $(0\,1\,\cdots\,n)$  into $n$ transpositions. To put our work in proper context we begin with a brief review of some well-known relationships among these  ubiquitous objects. Novel content commences in \Section{summary}, where we present an overview of our results.

\subsection{Notation}
For nonnegative integers $m,n\in \N$ we let $[n]:=\{0,1,\ldots,n\}$ and  $[m,n]  := \{ i \in \N \,:\, m \leq i \leq n\}$.   We write $\Sym{[n]}$ for the symmetric group  on $[n]$, and typically express its elements in cycle notation with fixed points suppressed, using  $\iota$ to denote the identity.  We caution the reader that we multiply permutations from \emph{left to right}, that is $\pi_1\pi_2\,:\, i \mapsto \pi_2(\pi_1(i))$.  Finally, the \emph{support} of a cycle $C=(c_1\,\cdots\,c_k)$ is denoted $\supp{C}:=\{c_1,\ldots,c_k\}$.

\subsection{Labelled Trees}

Let $\T_n$ be the set of all $(n+1)^{n-1}$  trees on vertex set $[n]$.  We regard each $T
\in \T_n$ as being rooted at 0, and say for distinct vertices $i$ and $j$ that $j$ is   a
\emph{descendant} of  $i$ in $T$ if $i$ lies on the unique path from $j$ to the root.   If
$j$ is a descendant of $i$, with $i > j$, then the pair $(i,j)$ is called an  \emph{inversion} in $T$.

Let $\inv(T)$ denote the number of inversions in $T$. The \emph{inversion enumerator} of $\T_n$  is defined by 
$$
	I_n(q) := \sum_{T \in \T_n} q^{\inv(T)}.
$$
These well-studied polynomials  were first introduced in~\cite{mallowsriordan}, where it was  shown that they satisfy the recursion
\begin{equation}
\label{eq:mallowsrecursion}
	I_{n+1}(q) = \sum_{i=0}^n \binom{n}{i}  (1+q+q^2+\cdots+q^i) I_i(q)  I_{n-i}(q).
\end{equation}
Kreweras~\cite{kreweras}  used this relation to establish
several interesting combinatorial properties of the sequence $\{I_n(q)\}$, including a link with parking functions to be described shortly. See also~\cite{beissinger,gesselwang} for  connections to the Tutte polynomial of the complete graph.

\subsection{Parking Functions}

A \emph{parking function} is a sequence $(a_1,\ldots,a_n) \in \N^n$ whose
non-decreasing rearrangement $(a_1', \ldots, a_n')$ satisfies $a_i' \leq i-1$
for all $i$.  Let $\p_n$ be the set of  parking functions of length $n$  and
let $\M_n$ be the set of  their complements $(n-a_1,\ldots,n-a_n)$, which are known as \emph{major sequences}.    Equivalently,
$(b_1,\ldots,b_n) \in \N^n$ is a major sequence if its  non-decreasing arrangement $(b_1',\ldots,b_n')$ satisfies $i \leq b_i' \leq n$ for all $i$.

Parking functions were first studied explicitly by Konheim and
Weiss~\cite{konheim}, who proved analytically that $|\p_n|=(n+1)^{n-1}$.
The equality $|\p_n|=|\T_n|$  anticipates  correspondences between parking functions and trees, and indeed   such bijections were found by various authors. The connection with trees was later refined by Kreweras as follows:
\begin{theorem}[\cite{kreweras}]\label{thm:kreweras}
For $n \geq 0$, we have
	\begin{equation*}
\label{eq:kreweras}
	I_n(q)
	= 
	\sum_{(a_1,\ldots,a_n) \in \p_n} q^{\binom{n}{2}-(a_1+ \cdots + a_n)}.
\end{equation*}
\end{theorem}
Kreweras proved Theorem~\ref{thm:kreweras} by showing  the polynomials on both sides of the identity  satisfy recursion~\eqref{eq:mallowsrecursion}. Several bijective proofs  have since been found, \emph{e.g.}~\cite{shin}.  We direct the reader to Yan's comprehensive survey~\cite{yansurvey}  for further details and references.

The  quantity $\binom{n}{2}-\sum_i a_i$ appearing in Theorem~\ref{thm:kreweras}
is called the \emph{area} of  the parking function $p=(a_1,\ldots,a_n)$. This terminology stems from an alternative view of parking functions
as labelled {Dyck paths}; that is, lattice paths from $(0,0)$ to $(n,n)$ that
remain weakly below $y=x$.   From the non-decreasing sequence $a_1' \leq \cdots
\leq a_n'$ corresponding to $p$, one draws a path $P$ whose $j$-th horizontal
step is at height $a_j'$. All horizontal steps at height $j$ are then labelled
with  $\{i : a_i = j\}$, in decreasing order left-to-right.     Evidently
$\binom{n}{2}-\sum a_i$ is the number of whole squares between $P$ and the diagonal $y=x$. See Figure~\ref{fig:factarea} (black). 
\begin{figure}[t]
	\centering
	\includegraphics[height=5cm]{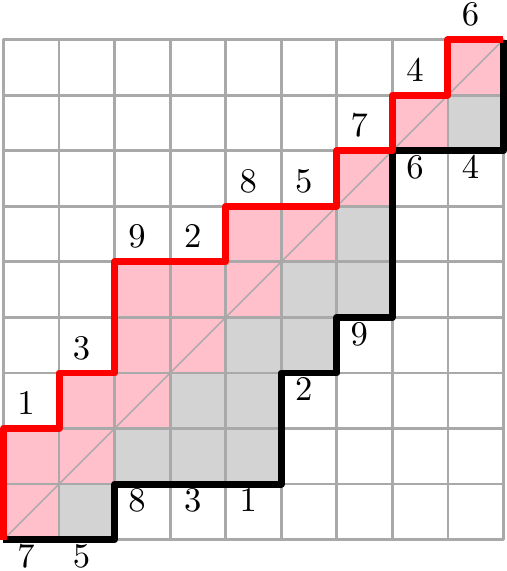}
	\caption{In black, the labelled Dyck path corresponding to
		$(1,3,1,7,0,7,0,1,4) \in \p_9$, with  shaded area
		$\binom{9}{2}-24=12$.  In red, the path corresponding
		to $(2,5,3,8,6,9,7,6,5) \in \M_n$, with shaded area
		$51-\binom{9}{2}=15$. 
	} 
	\label{fig:factarea}
\end{figure}

Similarly,  major sequences $q=(b_1,\ldots,b_n)$ correspond with labelled  paths remaining  weakly above $y=x$. The area of $q$ is defined to be the number of  squares under its path that lie \emph{on or above} the diagonal, \emph{i.e.} $\area(q):=\sum_i b_i - \binom{n}{2}$.    See Figure~\ref{fig:factarea} (red).

\subsection{Minimal Factorizations of Full Cycles}

Let $\cycles{n}$ denote the set of all full cycles in $\Sym{[n]}$. It is easy to see that any $\s \in \cycles{n}$ can be expressed as a product of $n$ transpositions and no fewer.  Accordingly, a sequence $(\t_1, \ldots,\t_n)$ of transpositions  satisfying  $\t_1\t_2\cdots\t_n=\s$ is called a \emph{minimal factorization} of $\s$.    Let $\F_\s$ be the set of all such factorizations.

We shall be particularly interested in factorizations of the canonical full cycle  $\s_n := (0\,1\,\cdots\,n)$.  For simplicity we write $\F_n$ in place of $\F_{\s_n}$. It is also convenient to express factorizations as products rather than tuples; for instance,
\begin{align*}
	\F_{1} &= \{(0\; 1)\},\\
	\F_{2} &= \{(0\; 1) (0 \; 2), (0\; 2) (1\; 2), (1\; 2)
	(0\; 1)\}.
\end{align*}  
We remind the reader that permutations are multiplied from left to right.

Minimal factorizations of $\s_n$ have long been known to be closely related to labelled trees.  The famous identity $|\F_n|=(n+1)^{n-1}$  dates back at least to Hurwitz but is often credited to D\'enes~\cite{denes}, who offered an elegant proof via indirect counting.  Direct bijections between $\F_n$ and $\T_n$  came later.  The simplest of these, due to Moszkowski~\cite{mosz}, has been rediscovered in different guises by a number of authors.

There is  a strikingly simple connection between factorizations and parking functions.
Consider the mappings $\map{L}{\F_n}{\N^n}$ and $\map{U}{\F_n}{\N^n}$  that send factorizations $f \in \F_n$ to their \emph{lower} and \emph{upper sequences}, respectively; that is,
\begin{align}
\label{eq:biane}
	(a_1\,b_1) \cdots (a_n\,b_n)
	\quad
	&\xmapsto{~L~}
	\quad
	(a_1,\ldots,a_n)  \\
	(a_1\,b_1) \cdots (a_n\,b_n)
	\quad
	&\xmapsto{~U~}	\quad
	(b_1,\ldots,b_n). \notag
\end{align}
Here, and throughout, we assume that indeterminate transpositions $(a\,b)$ are written such that $a < b$.   Then we have:

\begin{theorem}[{\cite{stanleyparking,biane}}]\label{thm:bianestanley}
The functions $L$  and $U$  map $\F_n$ bijectively into $\p_n$ and $\M_n$, respectively.
\end{theorem}

The bijectivity of $\map{L}{\F_n}{\p_n}$ is proved explicitly by Biane
\cite{biane}, though the author notes his result is equivalent to an earlier correspondence
of Stanley between parking functions and maximal chains in the lattice of
noncrossing partitions~\cite[Theorem 3.1]{stanleyparking}.  The bijectivity of
$\map{U}{\F_n}{\M_n}$ follows easily from that of $L$ via the involution
on $\F_n$ that first interchanges symbols $i$ and $n-i$ and  then reverses the order of the factors.

\section{Summary of  Results}
\label{sec:summary}

The  connections   between $\T_n, \F_n$ and $\p_n$ described in \Section{introduction} can be extended in  a few different directions. We will now summarize our contributions along these lines.  Proofs and various other details of the results stated here will be deferred to later sections, as indicated.

\subsection{Factorizations and Trees}

In light of Theorem~\ref{thm:bianestanley}, it is sensible to define the \emph{lower} and \emph{upper areas} of a factorization $f = \genericfact \in \F_n$ to be the areas of  $L(f) \in \p_n$ and $U(f) \in \M_n$, respectively. We write
\begin{align*}
	\AL(f) := \tbinom{n}{2}-(a_1+\cdots+a_n)
\intertext{and}
\AU(f) := (b_1+\cdots+b_n) - \tbinom{n}{2}
\end{align*}
for these quantities and further define
\begin{equation}
\label{eq:Fdefn}
	F_n(q,t) := \sum_{f \in \F_n} q^{\AL(f)} t^{\AU(f)}.
\end{equation}

Observe that Theorems~\ref{thm:kreweras} and~\ref{thm:bianestanley} immediately yield  $F_n(q,1)=I_n(q)$.  We seek to refine this identity by equating  $F_n(q,t)$ with a natural bivariate extension of the inversion enumerator.

Let us define a \emph{coinversion} in a tree $T \in \T_n$ to be a pair of vertices $(i,j)$ such that $j$ is a descendant of $i$ and $i < j$.
Thus every pair  of distinct vertices $(i,j)$ in $T$ with $j$ a descendant of $i$ is either an inversion $(i >
j)$, or a coinversion $(i < j)$.   Let $\ninv(T)$ denote the number of coinversions in $T$ and  set
$$
	I_n(q,t) := \sum_{T \in \T_n} q^{\inv(T)} t^{\ninv(T)}.
$$
The first few values of $I_n(q,t)$ are
\begin{align}\label{eq:exampinvenum}
	\begin{split}
	I_0(q,t) &= 1 \\
	I_1(q,t) &= t \\
	I_2(q,t) &= t^2 + t^3 + t^2 q\\
	I_3(q,t) &= t^6 + 2t^5q + 2t^4q^2 + t^3	q^3 + t^5 + t^4q + t^3q^2 + 3t^4
	+ 3t^3q + t^3.
\end{split}
\end{align}
Note that the asymmetry of $I_n(q,t)$ is a consequence of the root  0 contributing $n$ coinversions for every tree $T \in \T_n$.  The polynomial $t^{-n}I_n(q,t)$
counts only non-root coinversions and is seen to be symmetric by swapping  labels $i$ and $n-i+1$, for $i \neq 0$.\footnote{This asymmetry could naturally be remedied by  instead defining $I_n(q,t)$ over rooted forests on vertices $1,2,\ldots,n$.  However, the present definition is convenient for our purposes.}

Our first result is the following $(q,t)$-refinement of $|\T_n|=|\F_n|$. The case $n=3$  is illustrated in Figure~\ref{fig:mainresult}. 
\begin{theorem}
\label{thm:mainresult}
For $n \geq 0$ we have $I_n(q,t)=F_n(q,t)$.  That is, the bi-statistics $(\inv, \ninv)$ on $\T_n$ and $(\AL, \AU)$ on $\F_n$ share the same joint distribution.  
\end{theorem}
\begin{figure}[t]
	\centering
	\includegraphics[width=\textwidth]{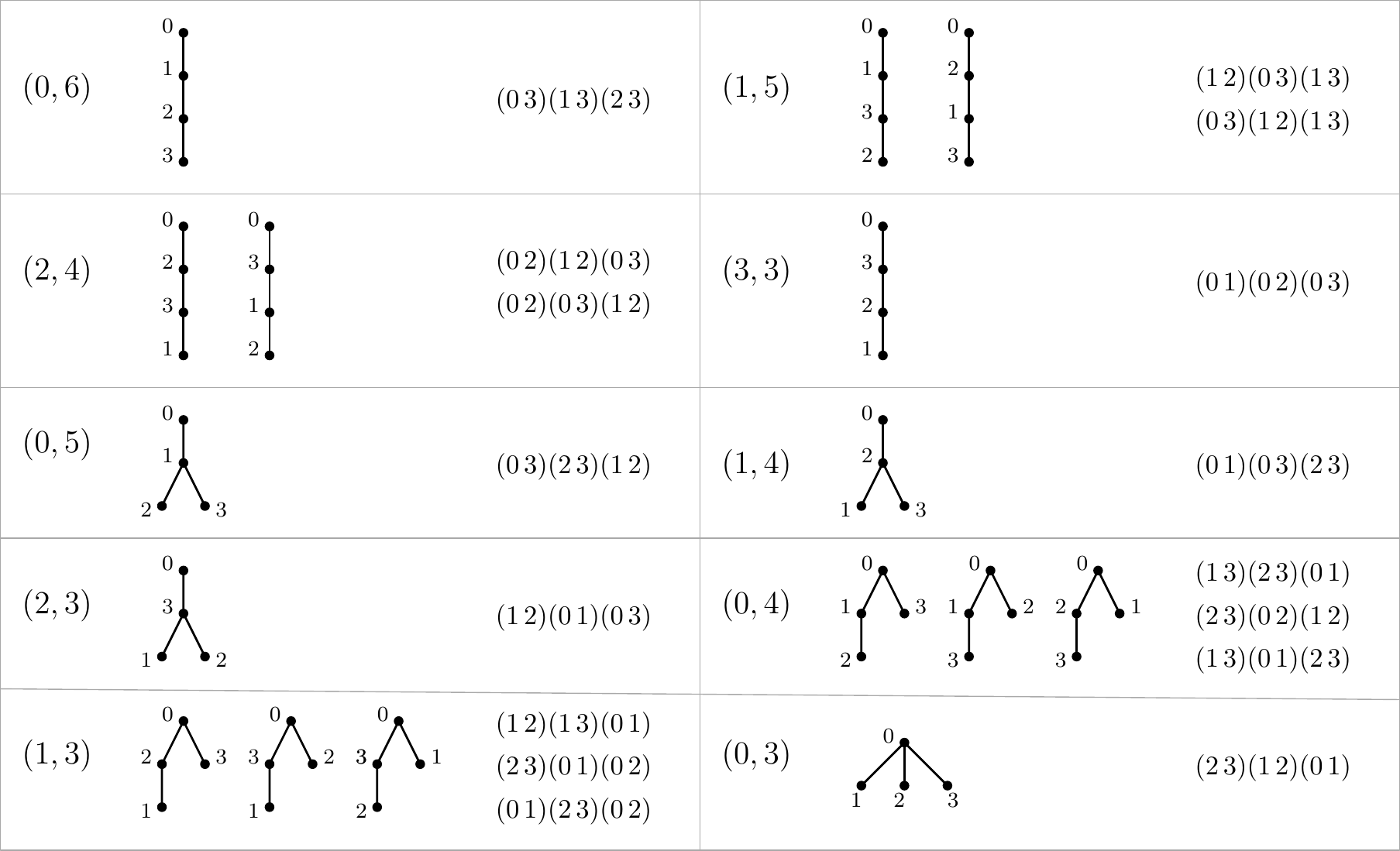}
	\caption{The sets $\T_3$ and $\F_3$ with elements grouped by common values of $(\inv,\ninv)$ and $(\AL, \AU)$.} 
	\label{fig:mainresult}
\end{figure}

Our proof of Theorem~\ref{thm:mainresult}   involves  verifying  that  $F_n(q,t)$ and
$I_n(q,t)$ both satisfy a straightforward refinement of a functional equation equivalent to 
recursion~\eqref{eq:mallowsrecursion}. Certain details of the proof  are simplified by employing  a   graphical model for factorizations.  The model will be developed in \Section{minimalfactorizations}, followed by the proof of the theorem in  Section~\ref{sec:mainthm}.  In \Section{special} we investigate the distribution of $(\AL, \AU)$ over several interesting subsets of $\F_n$.

\subsection{Factorizations and Parking Functions}
\label{sec:parkingfactorizations}

\newcommand{\Ls}{L_\s}
\newcommand{\Us}{U_\s}

For an arbitrary full cycle  $\s \in \cycles{n}$, one can define the lower and upper functions $\map{\Ls}{\F_\s}{\N^n}$ and $\map{\Us}{\F_{\s}}{\N^n}$ exactly as in~\eqref{eq:biane}. It is not difficult to show that $\Ls$ and $\Us$ always map $\F_\s$ into $\p_n$ and $\M_n$, respectively.  It is then  natural to ask for which $\s$  these mappings are bijective.  Our next result generalizes Theorem~\ref{thm:bianestanley} by  characterizing all such $\s$.

Let us say  the full cycle  $(0\,s_1\, \cdots \,s_n) $ is \emph{unimodal} if there is some $N$ such that $s_N=n$ and $0 < s_1 < s_2 < \cdots < s_N > s_{N+1}  > \cdots > s_n$. For example, $(0\,2\,3\,5\,4\,1) \in \cycles{5}$ is unimodal whereas $(0\,1\,4\,3\,5\,2)$ is not.  
Note that there are $2^{n-1}$ unimodal full cycles in $\Sym{[n]}$, since each is determined  by a subset $\{s_1,\ldots,s_{N-1}\}$ of $[1,n-1]$.  Then:

\begin{theorem}
For any $\s\in \cycles{n}$ we have $L(\F_\s)\subseteq \p_n$ and $U(\F_\s) \subseteq \M_n$.  Moreover,  the following are equivalent:
\begin{enumerate}
\item	$\s$ is unimodal.
\item	$\map{\Ls}{\F_\s}{\p_n}$ is a bijection.
\item	$\map{\Us}{\F_\s}{\M_n}$ is a bijection.
\end{enumerate}
	\label{thm:unimodal}
\end{theorem}

Theorem~\ref{thm:unimodal} is proved in \Section{unimodal}.  Notably,
our proof  entails an explicit  construction of
$\map{\Ls^{-1}}{\p_n}{\F_\s}$ in terms of  a simple graphical
algorithm.  Even in the case $\s=\s_n$ this algorithm for constructing
a factorization from a parking function appears to fill a  gap in the literature.

If $\s \in \cycles{n}$ is unimodal, then  Theorem~\ref{thm:unimodal} associates every factorization $f \in \F_\s$ with a unique pair $(P_L, P_U)$ of labelled
Dyck paths,  where $P_L$ corresponds to $\Ls(f) \in \p_n$ and $P_U$ to  $\Us(f)  \in  \M_n$.  Figure~\ref{fig:factarea} in fact shows the paths corresponding to the factorization  
\begin{equation}\label{eq:factinfig}
(1\,2)(3\,5)(1\,3)(7\,8)(0\,6)(7\,9)(0\,7)(1\,6)(4\,5) \in \F_9.
\end{equation}

We refer to $P_L(f)$ and $P_U(f)$ as the \emph{lower} and \emph{upper paths} of $f$.   While each of these paths is determined by the other, it is by no means obvious from  a diagram like Figure~\ref{fig:factarea} how one might  go about reconstructing $P_U(f)$ from $P_L(f)$.   In the case $\s = \s_n$ our algorithmic definition of $\map{L^{-1}}{\p_n}{\F_n}$  affords a particularly simple description of this reconstruction.  See \Section{unimodal} for details.

A final note in this section is to ask, in light of Theorem \ref{thm:unimodal}, whether
a result analogous to Theorem \ref{thm:mainresult} 
applies to unimodal cycles in general;  that is, if the analogous polynomial
in \eqref{eq:Fdefn} for an arbitrary unimodal cycle equals $I_n(q,t)$.

The highest and lowest degree terms in $I_3(q,t)$ are of degree 6 and 3, respectively, as
seen in the explicit expressions in \eqref{eq:exampinvenum}.  But a quick check will show that the unimodal cycle $\s = (0\, 1\, 3\, 2)$
has the two factorizations
\begin{align*}
	(0\, 1\, 3\, 2) &= (0\, 3) (1\, 3) (0\, 2)\\
	&= (1\, 2) (0\, 1) (2\, 3).
\end{align*}
The first factorization has lower and upper areas equal to 2 and 5,
respectively, summing to 7, while the second factorization has lower and upper areas equal
to 0 and 3, respectively.  Therefore, the polynomial $\F_{\s}(q,t)
= \sum_{\fact \in \F_{\s}} q^{\AL(\fact)} t^{\AU(\fact)}$ has terms $q^2t^5$ (which has
degree 7) and $t^3$.  It follows that $\F_{\s}(q,t)$ is neither equal to nor
multiple of $I_3(q,t)$.  An open question is to characterize which unimodal
cycles have lower and upper polynomials equal to $I_n(q,t)$.

\subsection{Depth, Difference, and Bounce}
\label{sec:ddb}

Specializing Theorem~\ref{thm:mainresult} at $q=t$ gives $I_n(q,q)=F_n(q,q)$.  
Let $\depthGS_n(q)$ be the common value of these series; that is,
\begin{equation}
\label{eq:depth}
	\depthGS_n(q) := \sum_{T \in \T_n} q^{\inv(T)+\ninv(T)}
	= \sum_{f \in \F_n} q^{\AL(f)+\AU(f)}.
\end{equation}

Observe that $\inv(T)+\ninv(T)$ is the number of pairs $(i,j)$ of distinct
vertices in $T \in \T_n$ such that $j$ is a descendant of $i$. This is easily
seen to be the sum of the distances from the root to all other vertices, a
quantity we call the \emph{total depth}\footnote{Also  known in the literature as the \emph{path
length} of $T$.} of $T$ and denote  by $\tdep(T)$.  Interpreting $D_n(q)$ in this way, it  is  routine  to show that  $D(x) = \sum_n D_n(q) \frac{x^n}{n!}$ satisfies $D(x) = \exp( qxD(qx))$.  The initial values of $D_n(q)$ are found to be $D_0(q)=1, D_1(q)=q$ and
\begin{align*}
	D_2(q) &= q^2(1+2q) \\
	D_3(q) &= q^3(1+6q+3q^2+6q^3) \\
	D_4(q) &= q^4(1+12q+24q^2+28q^3+24q^4+12q^5+24q^6).
\end{align*}

Returning to~\eqref{eq:depth}, note that  for $f = \genericfact \in
\F_n$  we have $\AL(f)+\AU(f)=\sum_i (b_i-a_i)$. We call this quantity
the  \emph{total difference} of $f$, written $\diff{f}$.   From the
graphical point of view,  $\diff{f}$ has the natural interpretation as the total area
bounded between the upper and lower paths of $f$.  The reader is again referred to
Figure \ref{fig:factarea} with reference to factorization
\eqref{eq:factinfig}.

Besides counting trees by total depth and factorizations by total difference,  it transpires that $\depthGS_n(q)$ also enumerates parking functions with respect  to Haglund's  ``bounce'' statistic~\cite{hag}, whose definition we now recall.

\label{page:bounceintro}

Let $P$ be the  path corresponding to  $p=(a_1, \dots, a_n) \in \p_n$. Imagine a ball starting at  the origin and moving east until it encounters a vertical step of $P$, at which time it ``bounces'' north and continues until it encounters the line $y=x$, where it bounces east until it hits $P$, etc.  The ball continues to bounce between $P$ and the diagonal until it arrives  
at $(n,n)$, as shown in Figure~\ref{fig:bouncepath}.  Let $0 = i_1 < \cdots < i_k = n$ be the $x$-coordinates at which the ball meets $y=x$.  Then the \emph{bounce} of $p$ is defined by 
\begin{equation}\label{eq:defbounce}
	\bnce(p) := \sum_j (n-i_j).
\end{equation}
\begin{figure}[t]
	\centering
	\includegraphics[height=5cm]{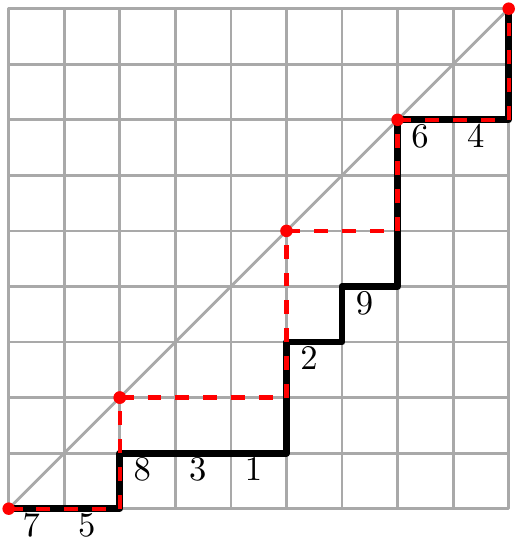}
	\caption{The bounce path of  ${p=(1,3,1,7,0,7,0,1,4) \in \p_9}$ is shown in red. This path meets the diagonal at $x=0,2,5,7,9$, giving $\bnce(p)=(9-0)+(9-2)+(9-5)+(9-7)+(9-9)=22$.} 
	\label{fig:bouncepath}
\end{figure}

Note that  the labels on $P$ are immaterial in the definition of $\bnce(p)$. Indeed,  bounce is truly a statistic on Dyck paths, not parking functions.  Nonetheless, the distribution of $\bnce(p)$ over $\p_n$ is of interest:

\begin{theorem}
\label{thm:bounce}
For any $k \in \N$, the following objects are equinumerous: (1) trees $T \in \T_n$ of total depth $k$, (2) minimal factorizations $f \in \F_n$ with total difference $k$, and (3) parking functions $p \in \p_n$ whose  bounce is $k$. 
That is, we have
$$
	D_n(q) =
	\sum_{T \in \T_n} q^{\tdep(T)} 
	= 
	\sum_{f \in \F_n} q^{\diff{f}}
		=
	\sum_{p \in \p_n} q^{\bnce(p)}.
$$
\end{theorem}

Theorem~\ref{thm:bounce} will be proved in \Section{bounce} as a corollary of a more refined connection between bounce and  total depth of trees.  The relation between bounce and total difference comes indirectly via Theorem~\ref{thm:mainresult}.  In particular, we caution the reader that it is \emph{not} generally  the case that $\bnce(L(f))=\diff{f}$.

\subsection{Further Comments}

Given Theorem~\ref{thm:mainresult}, one might expect Kreweras' formula (Theorem~\ref{thm:kreweras}) to admit a refinement of the form
$$
	I_n(q,t)
	= \sum_{p \in \p_n} q^{\area(p)} t^{\mathrm{stat}(p)},
$$
where $\mathrm{stat}(\cdot)$ is some natural statistic on $\p_n$.   Indeed, such a statistic can be defined  in terms of the ``parking processes'' for which  parking functions are named.

Consider a doubly infinite parking lot consisting of stalls labelled $\ldots,-2,-1,0,1,2,3,\ldots$ from west to east.  A total of $n$ cars wish to park in the lot, one after the other.   The $i$-th car enters the lot at stall $a_i$ and either parks there, if empty, or continues to the  first unoccupied spot  to the east.  It is well known that the  sequence $(a_1,\ldots,a_n) \in \N^n$ is a parking function if and only if the cars occupy stalls $0$ through $n-1$ after all are parked~\cite{stanleyparking}. 

Suppose $p=(a_1,\ldots,a_n) \in \p_n$ and let $c_i$ be the stall in
which car $i$   parks according to this procedure.  Then
$(c_1,\ldots,c_n)$ is a permutation of $[n-1]$, and $c_i-a_i$ is the
number of  stalls  the $i$-th car must ``jump``  before finding a place to park.   Following Shin~\cite{shin}, define 
$
	\jump{p}:=\sum_i (c_i - a_i)
$    
and observe that 
$$
\jump{p} = (0+1+\cdots+n-1)-(a_1+\cdots+a_n)=\area(p).
$$   Shin proves Theorem~\ref{thm:kreweras} by giving a bijection  $\map{\phi}{\T_n}{\p_n}$ such that $\inv(T)=\jump{\phi(T)}$.   

Now consider the state of the lot after the first $i$ cars have
parked. Let $d_i$ be the closest empty stall  \emph{west} of car $i$
at this point.  Thus $a_i-d_i$ is the number of stalls a  car would
jump  if it were to enter at $a_i$ and take the first available spot
to the west.     If we define $\cojump{p} := \sum_i (a_i-d_i)$, then it is straightforward to deduce from Shin's argument that  $\ninv(T)=\cojump{\phi(T)}$.  There follows 
\begin{equation}
\label{eq:shin}
	I_n(q,t) = \sum_{p \in \p_n} q^{\jump{p}} t^{\cojump{p}}.
\end{equation}

At first glance, our proof of Theorem~\ref{thm:mainresult} and the
proof of~\eqref{eq:shin} described above appear to be quite different.
However,  these   results can be unified by viewing
$(\mathrm{jump},\mathrm{cojump})$ and $(\AL, \AU)$ as Mahonian
bi-statistics on trees. Just as $(\jump{p},\cojump{p})$ corresponds
with $(\inv(T),\ninv(T))$, we have found that $(\AL(f),\AU(f))$
corresponds with $(\mathrm{maj}(T),\mathrm{comaj}(T))$,  the
\emph{major/comajor index} of $T$.    Moreover, we have discovered this approach  is
readily generalized, allowing trees to be replaced by  $k$-cacti and
permitting various other refinements.   These results on the major indices, their
connections to factorizations and their $k$ generalizations were presented at the CanaDAM 2019 conference and 
will be exposited in detail in the forthcoming
article \cite{irkpaper} (in preparation).

\section{Minimal Factorizations and Arch Diagrams}
\label{sec:minimalfactorizations}

\newcommand{\graph}[1]{G(#1)}

Let $\pi$ be an arbitrary permutation. 
A \emph{factorization} of $\pi$ is any sequence $f=(\t_1,\ldots,\t_m)$ of transpositions satisfying $\t_1\cdots\t_m=\pi$. We say $f$ is \emph{minimal} if there is no factorization of $\pi$ having fewer than $m$ factors.

Observe that factorizations in $\Sym{[n]}$ naturally correspond with edge-labelled graphs on  $[n]$.  The \emph{graph} of $f=(a_1\,b_1)\cdots(a_m\,b_m)$,  denoted $\graph{f}$, has an edge $a_i b_i$ labelled $i$ for each $i \in [1,m]$. 
See Figure~\ref{fig:archdiagram} (left).
\begin{figure}[t]
		\centering
		\includegraphics[width=.9\textwidth]{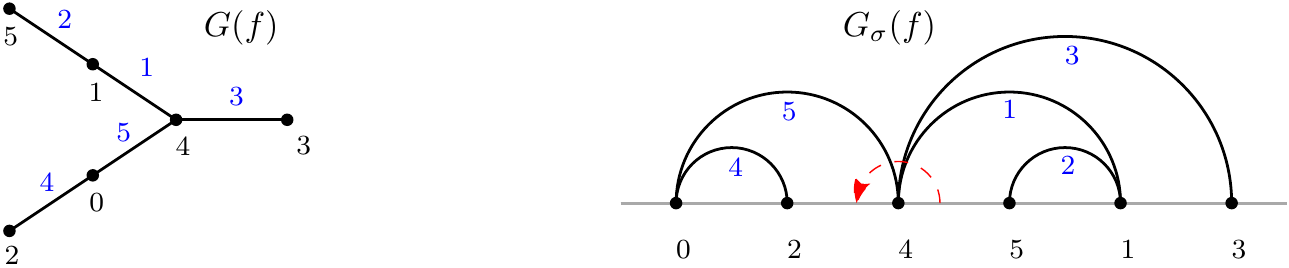}
		\caption{(left) The graph of $f=(1\,4)(1\,5)(3\,4)(0\,2)(0\,4) \in \F_{\s}$, where $\s=(0\,2\,4\,5\,1\,3)$. (right) The $\s$-diagram of $f$. The rotator of vertex 4 is $(1,3,5)$.}
		\label{fig:archdiagram}
\end{figure} 

Our aim in this section is to associate  every $f \in \F_{\s}$ with a particular planar embedding of  $\graph{f}$.  For completeness we first establish a well-known graphical characterization of minimality.

Consider the effect of multiplying a permutation $\rho$ by a transposition
$\t=(a\,b)$. If $a,b$ lie on the same cycle of $\rho$, then this cycle is split
into two cycles of the product $\rho\t$ (one containing $a$ and the other $b$).
In this case we say $\t$ \emph{cuts} $\rho$.  If  $a,b$ instead lie on
distinct cycles of $\rho$, then these are merged into one cycle of
$\rho\t$ and we say $\t$ \emph{joins} $\rho$.  The factor $\t_i$ is a
\emph{join} (respectively, \emph{cut}) of the factorization $f =
(\t_1,\ldots,\t_m)$  if it joins (cuts)  the partial product $\t_1
\cdots \t_{i-1}$.

Suppose the cycles of $\pi$ are of lengths $n_1,\ldots,n_k$. Then
clearly $\pi$ can  be factored into $\sum_i (n_i-1)=n+1-k$
transpositions.  But any factorization of $\pi$ must contain at least
this many  joins, since the effect of multiplying its factors is to
merge $n+1$ trivial cycles of the identity into  $k$ cycles of $\pi$.
A factorization of $\pi$ is therefore minimal if it is of length $n+1-\len{\pi}$, where $\len{\pi}$ denotes the number of cycles of $\pi$.  Equivalently, $f$ is minimal if each  of its factors is a join.

Let $f=(\t_1,\ldots,\t_m)$ be a minimal factorization and, for $i \in
[n]$, let $G_i$ be the spanning subgraph of $\graph{f}$ containing
only those edges with label $\leq i$. Thus, $G_0$ is the empty graph on
$[n]$, and $G_i$ is obtained from $G_{i-1}$ by adding edge $a_ib_i$, where $\t_i=(a_i\,b_i)$. Since all factors of $f$ are joins, we see inductively that
each cycle of   $\t_1 \cdots \t_{i-1}$ is supported by the vertices of some
component of $G_{i-1}$. Since $a_i$ and $b_i$ belong to different  cycles of $\t_1 \cdots \t_{i-1}$ (\emph{i.e.} different components of $G_{i-1}$) we conclude that $G_i$ is a forest. This leads to:

\begin{lemma}
\label{lem:minimal1}
Let $f$ be a factorization of $\pi \in \Sym{[n]}$. 
Then $f$ is minimal if and only if $G(f)$ is a forest consisting of $\len{\pi}$ trees, each of whose sets of vertices support some cycle of $\pi$.  In particular, $f$ is a minimal factorization of a full cycle if and only if $\graph{f}$ is a tree.
\end{lemma}

\newcommand{\mathsc}[1]{{\normalfont\textsc{#1}}}
\newcommand{\chord}[1]{\mathsc{chord}(#1)}
\newcommand{\arch}[1]{\mathsc{arch}(#1)}
\newcommand{\archst}[1]{\mathsc{arch}'(#1)}
\newcommand{\FACT}[1]{\mathsc{fact}(#1)}
\newcommand{\FACTs}[2]{\mathsc{fact}_{#1}(#2)}
\newcommand{\Edge}[1]{\mathcal{T}_{#1}}
\newcommand{\TVE}[2]{\mathcal{T}(#1,#2)}
\newcommand{\cd}{c.d.}
\newcommand{\cds}{c.d.'s}
\newcommand{\archset}[1]{\mathcal{A}_{#1}}
\newcommand{\Trees}[2]{\mathcal{T}(#1,#2)}
\newcommand{\hatarchset}[1]{\hat{\mathcal{A}}_{#1}}
\newcommand{\Sset}[1]{\mathcal{S}_{#1}}
\newcommand{\sdiagram}[2]{G_{#1}(#2)}

Let $\s=(s_0\,s_1\,\cdots\,s_n) \in \cycles{n}$, with $s_0=0$, and let $f$ be any factorization in $\Sym{[n]}$.  The $\emph{$\s$-diagram}$ of $f$, denoted $\sdiagram{\s}{f}$, is a drawing of
$\graph{f}$ in the plane in which vertex $s_i$ lies at the point $(0,i)$ and edges are rendered as semicircular arches above the $x$-axis.  The \emph{rotator} of a vertex $v$ in $\sdiagram{\s}{f}$ is the ordered list of arch labels encountered on a counter-clockwise tour about $v$ beginning on the axis.

Figure~\ref{fig:archdiagram} (right) displays $\sdiagram{\s}{f}$ for a minimal factorization $f$ of the full cycle $\s=(0\,2\,4\,5\,1\,3)$. Observe that it is a tree, in accordance with Lemma~\ref{lem:minimal1}. Moreover, its edges meet only at  endpoints and all its rotators are increasing. It turns out that these conditions  characterize the set $\{\sdiagram{\s}{f} \,:\, f \in \F_\s\}$. 

\begin{theorem}
\label{thm:gy}
Let $f$ be a factorization in $\Sym{[n]}$ and let $\s \in \cycles{n}$.  Then $f \in \F_\s$ if and only if  $\sdiagram{\s}{f}$ is a planar embedding of a tree in which every rotator is increasing.
\end{theorem}

The  case $\s=\s_n$ of Theorem~\ref{thm:gy} can be found in~\cite{gouldenyong},  albeit phrased somewhat differently in terms of ``circle chord diagrams''. (These are obtained from $\s$-diagrams by wrapping the $x$-axis into a circle so that arches become chords.) The general case follows immediately  by relabelling. We mention in passing that Theorem~\ref{thm:gy} is but one manifestation of a  general equivalence between factorizations in $\Sym{[n]}$ and  embeddings of graphs on surfaces.  There is a vast literature on this subject; see for example the surveys~\cite{lando-zvonkin,gouldensurvey} and   references therein.

\newcommand{\archx}[1]{\mathsc{arch}^*(#1)}

In light of Theorem~\ref{thm:gy}, we define an \emph{arch diagram} to be any noncrossing embedding of an edge-labelled tree  into the upper half-plane $\{(x,y):y \geq 0\}$ such that  vertices lie on the $x$-axis  and rotators are increasing.  Let  $\archset{n}$ denote the set of topologically inequivalent arch diagrams with edges  labelled $1,\ldots,n$.

Let $f \in \F_\s$ for some $\s \in \cycles{n}$. Then stripping $\sdiagram{\s}{f}$ of its vertex labels leaves an arch diagram $\mathsc{arch}_\s(f) \in \archset{n}$.  In fact, since the missing labels can be recovered from $\s$,   Theorem~\ref{thm:gy} implies  $f \mapsto \mathsc{arch}_\s(f)$ is a bijection between $\F_\s$ and $\archset{n}$.  The inverse map $\map{\mathsc{fact}_\s}{\archset{n}}{\F_\s}$ is illustrated in Figure~\ref{fig:archfact}. For simplicity we shall write $\mathsc{arch}$ and $\mathsc{fact}$ in place of $\mathsc{arch}_{\s_n}$ and $\mathsc{fact}_{\s_n}$ when $n$ is understood from context. 
\begin{figure}[t]
		\centering
 \includegraphics[width=.95\textwidth]{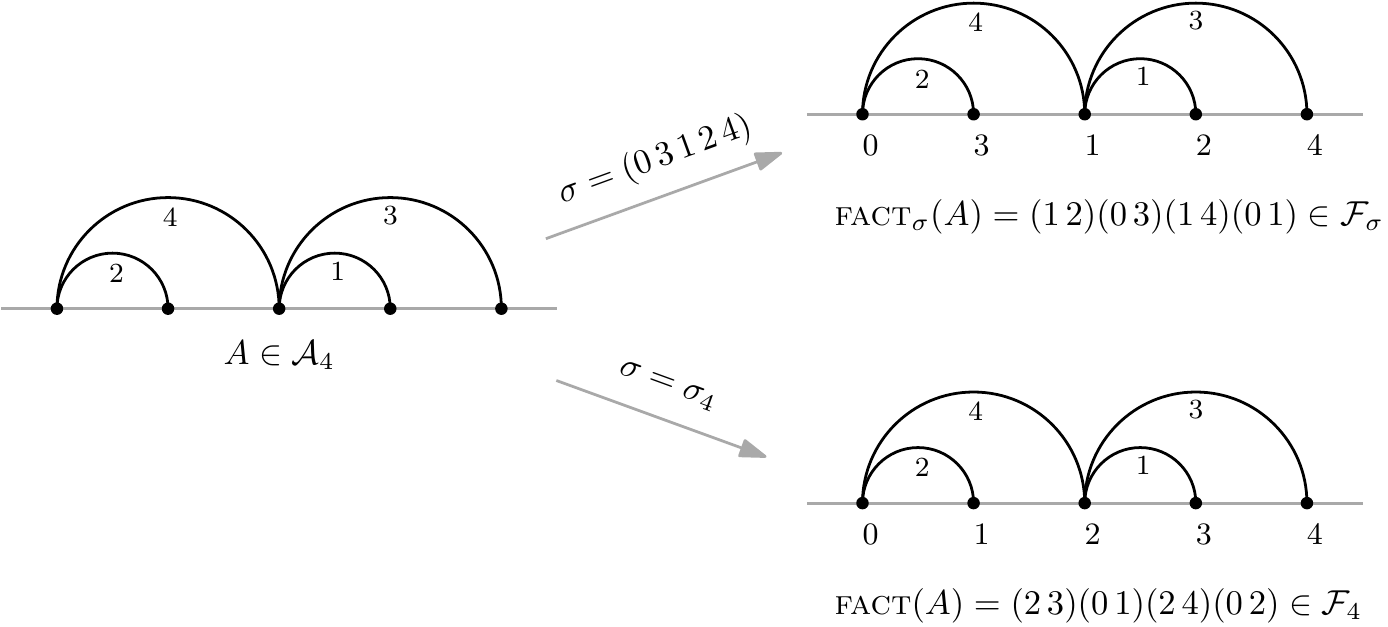}
		\caption{The correspondence between $\archset{n}$ and $\F_{\s}$.}
		\label{fig:archfact}
\end{figure}

\ignore{
\newcommand{\rootset}[1]{\mathcal{E}_{#1}}
\begin{remark}
Let $v$ be any vertex of an edge-labelled tree $T$. Observe that there is a unique arch diagram skeletal graph $T$ and rightmost vertex $v$.  and $v$ as its rightmost vertex. Indeed, once $v$ is embedded as the rightmost vertex, the relative ordering along the axis of all other vertices is forced by the increasing rotator condition. 

Let $\rootset{n}$ be the set of vertex rooted trees on $n+1$ unlabelled vertices with edges distinctly labelled $1,\ldots,n$.   The preceding comments establish an equivalence between $\archset{n} \leftrightarrow \rootset{n}$, wherein the rightmost vertex of $A \in \archset{n}$ is viewed as the root of its skeletal tree. 

There is  a simple bijection $\rootset{n} \rightarrow \T_n$ under which the root of $T \in
\rootset{n}$ is assigned label 0 and edge labels are ``pushed'' away from the root onto
the remaining vertices. The composite mapping $\F_n \xrightarrow{\mathsc{arch}}
\archset{n} \rightarrow \rootset{n} \rightarrow \T_n$  is a variant of Moszkowski's
bijection~\cite{mosz}. See Figure~\ref{fig:moz}.\attentionpreet{I think you have a reason
for changing the root (i.e. why not just keep 0 as the 0 vertex?).  Just checking.  if you
don't have a reason, we may consider changing this.}
\begin{figure}[t]
\centering
\includegraphics[width=.85\textwidth]{moszkowski2.pdf}
\caption{A variant of Moszkowski's bijection, mapping a factorization $f \in \F_5$ to a tree $T \in \T_5$.}
\label{fig:moz}
\end{figure} 
}

\section{Factorizations and Trees}
\label{sec:mainthm}

In this section we shall  prove Theorem~\ref{thm:mainresult}  by showing that the polynomial sequences $\{I_n(q,t)\}$ and $\{F_n(q,t)\}$ satisfy a common recursion defined in terms of their exponential generating series
\begin{equation}
	F(x; q, t)  = \sum_{n \geq 0} F_n(q,t) \frac{x^n}{n!}\,\,\,\,\,\, \textnormal{ and
	}\,\,\,\,\,\, I(x; q, t)  = \sum_{n \geq 0} I_n(q,t) \frac{x^n}{n!}.
	\label{eq:deffi}
\end{equation}

\newcommand{\qt}[1]{(t^{#1-1}+t^{#1-2}q + t^{#1-3}q^2+\cdots + q^{#1-1})}

The recursion for $I_n(q,t)$ comes via the familiar decomposition of a tree $T \in \T_n$  into a set of rooted labelled trees by deletion of its root.   Carefully accounting for inversions and coinversions  under this decomposition yields
\begin{equation}
	I(x; q, t) = \exp \left( t \sum_{n=1}^\infty \qt{n}
	I_{n-1}(q,t) \frac{x^n}{n!}\right), \label{eq:expdecomp}
\end{equation}
which is readily seen to be equivalent to 
\begin{equation}
\label{eq:recursion}
	I_{n+1}(q,t) = \sum_{i=0}^n \binom{n}{i} t (t^i+t^{i-1}q+t^{i-2}q^2+\cdots+q^i) I_i(q,t)  I_{n-i}(q,t).
\end{equation}
Note that~\eqref{eq:recursion} reduces to~\eqref{eq:mallowsrecursion} at $t=1$. 
The proofs of~\eqref{eq:expdecomp} and~\eqref{eq:recursion} at $t=1$ can be found in~\cite[Theorem 3]{gessel} or~\cite[Theorem 1]{kreweras}.  Proofs for arbitrary $t$
require only trivial modifications to their proofs and are omitted here.

Since $F_0(q,t)=I_0(q,t)=1$, Theorem~\ref{thm:mainresult} follows immediately from the following analogue of~\eqref{eq:expdecomp} for minimal factorizations.

\begin{theorem}\label{thm:mainthm}
	The series $F(x;q,t)$ satisfies
	\begin{equation*}
	F(x;q,t) = \exp \left( t \sum_{n=1}^\infty \qt{n}
	F_{n-1}(q,t) \frac{x^n}{n!}\right).	
	\end{equation*}
\end{theorem}

We now prove Theorem \ref{thm:mainthm} in two stages. We first show that its exponential nature reflects a canonical splitting of factorizations $f \in \F_n$ into sets of ``simple'' factorizations. We then describe a correspondence between simple factorizations in $\F_n$ and arbitrary factorizations in $\F_{n-1}$. 

Say a factorization $f \in \F_n$ is \emph{simple} if it contains the factor $(0\,n)$.  
Let $\hatfset_{n}$ be the set of simple factorizations in $\F_n$ and define
\begin{equation*}
	\hatffunc_n(q,t) := \sum_{\fact \in \hatfset_{n}} t^{\AU(\fact)} q^{\AL(\fact)}.
\end{equation*}
Evidently $f$ is simple if and only if  the left- and rightmost vertices of $\arch{f}$  are joined by an edge, so we say an arch diagram is \emph{simple} if it has this property. Let $\hatarchset{n}=\arch{\hatfset_n}$ be the set of all simple diagrams in $\archset{n}$. 

A \emph{cap} of an arch diagram is an edge that is not nested within any other edge.    For example, the 
diagrams in Figure \ref{fig:archfact} each have two caps, with labels 3 and
4. Clearly every arch diagram has at least one
cap, and $A \in \archset{n}$ is simple if and only if it has a unique cap.

\begin{proposition}
	$F(x;q,t) = \exp \left( \sum_{n \geq 1} \hatffunc_n(q,t) \frac{x^n}{n!} \right).$
\label{thm:exponential}
\end{proposition}

\begin{proof}
We employ a natural  decomposition of noncrossing arch diagrams into  sets of simple diagrams. Figure~\ref{fig:decomposition} illustrates the mechanics of the proof.
	\begin{figure}[t]
		\centering
		\includegraphics[width=.9\textwidth]{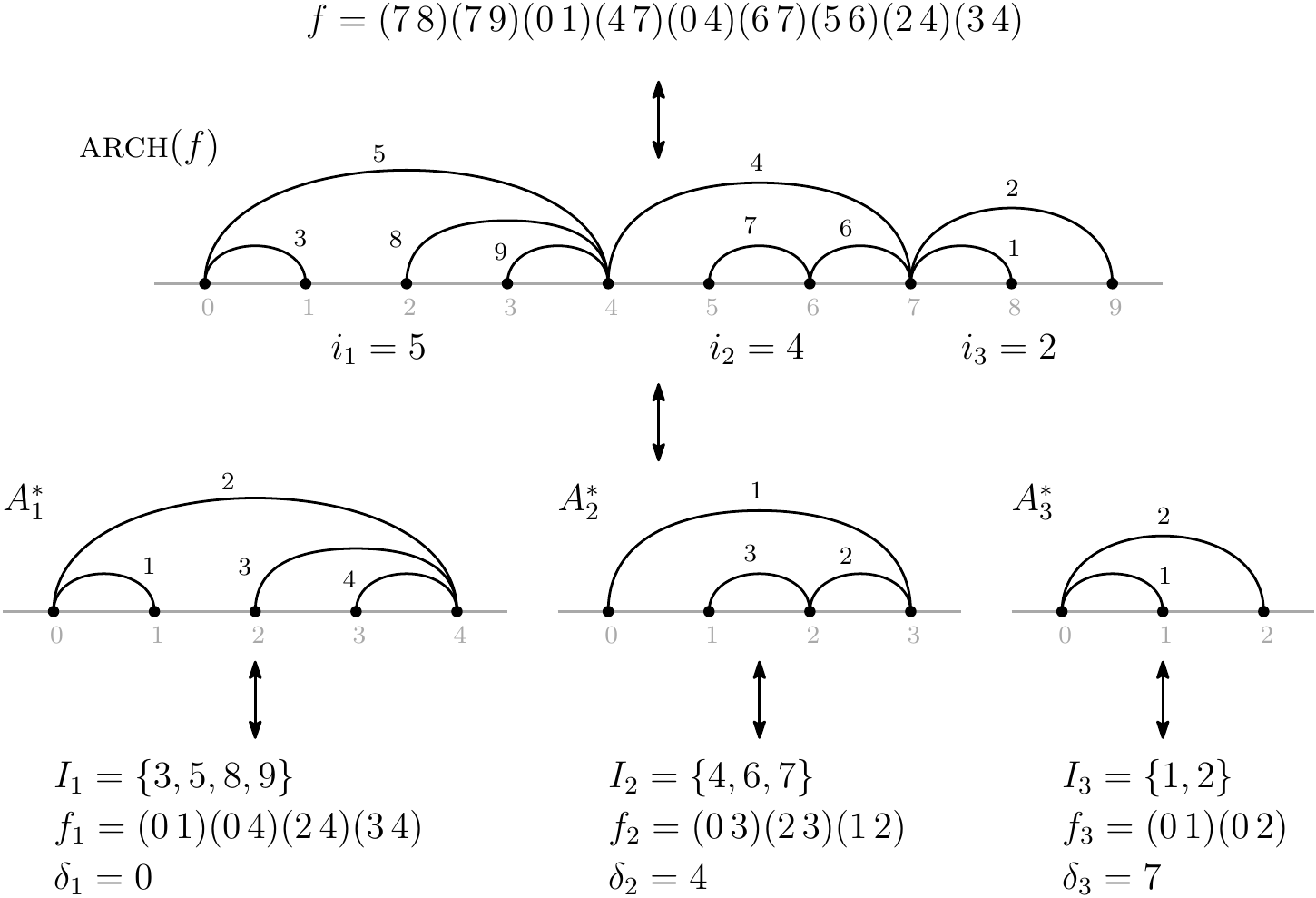}
		\caption{An illustration of the proof of
Proposition~\ref{thm:exponential}. The factorization $f \in \F_9$ naturally decomposes via
its arch diagram into simple factorizations $f_j \in \hatfset_{|I_j|}$, $j=1,2,3$, where
$\{I_1,I_2,I_3\}$  partitions $[1,9]$.  } 
		\label{fig:decomposition}
	\end{figure}

Let $f \in \F_n$, where $n \geq 1$, and let $A = \arch{f}$.
Let $e_1,\ldots,e_k$ be the caps of $A$ and let $i_1,\ldots,i_k$ be their labels.
Then $\{e_1,\ldots,e_k\}$ is the edge-set of a path from the leftmost to the
rightmost vertex of $A$. For  $j \in [1,k]$,  let $A_j$ be the arch diagram induced by all
edges of $A$ nested under $e_j$, including $e_j$ itself. Evidently $A_j$ is simple, with
$e_j$ being  its unique cap. Preserve the edge labels of $A_j$ in the set $I_j$, and then
replace the $i$-th largest label of $A_j$ with $i$ (for all $i$) to get an
arch diagram $A_j^*$ with edge labels  $\{1, 2, \ldots, |I_j|\}$.

Thus we have a decomposition $A \mapsto \{(A_1^*,I_1),\ldots,(A_k^*,I_k)\}$,  where
$\{I_1,\ldots,I_k\}$ is a partition of $[1,n]$ and $A_j^* \in
\hatarchset{|I_j|}$. We claim this is reversible.
Indeed,  $A_j$ and its original cap label $i_j$ are easily reconstructed from $(A_j^*, I_j)$, while increasing rotators force the $A_j$'s to be concatenated in decreasing left-to-right order of their cap labels $i_j$.

Let $f_j = \FACT{A_j^*} \in \hatfset_{|I_j|}$. We claim that $\AL(f) = \sum_j \AL(f_j)$ and $\AU(f) = \sum_j \AU(f_j)$. The result then follows from Theorem~\ref{thm:gy} and the exponential formula~\cite[Corollary 5.1.6]{stanenum2}.

To prove the claim, assume without loss of generality that $i_1 > \cdots > i_k$, which means the $A_j$'s are concatenated in left-to-right order $A_1,\ldots,A_k$ to form $A$. Let $\delta_j = |I_1|+\cdots+|I_{j-1}|$, with $\delta_1=0$. Then the $r$-th vertex of $A_j$ (counting from the left) corresponds to the $(r+\delta_j)$-th vertex of $A$. Thus each factor $(a\,b)$ of $f_j$ corresponds to a factor $(a+\delta_j\; b+\delta_j)$ of $f$.   Letting $L(\cdot)_i$ denote the $i$-th term of the lower sequence $L(\cdot)$, we have
\begin{align*}
\sum_{j=1}^k \AL(f_j)
&= \sum_{j=1}^k \sum_{i=1}^{|I_j|} \Big( i -1 - L(f_j)_i \Big) \\
&= \sum_{j=1}^k \sum_{i=1}^{|I_j|} (\delta_j+i-1) - 
\sum_{j=1}^k \sum_{i=1}^{|I_j|} (L(f_j)_i+\delta_j)  \\
&= \sum_{i=1}^n (i-1) - \sum_{i=1}^n L(f)_i  \\
&= \AL(f).
\end{align*}
A similar computation shows $\AU(f) = \sum_j \AU(f_j)$.  
\end{proof}

\begin{lemma}
	Let $\fact = (\tau_1, \dots, \tau_n) \in \hatfset_{n}$. Then there is a unique $k$ such that $\tau_k = (0\,n)$. Moreover, $\t_k$  is the rightmost factor moving 0 and the leftmost factor moving $n$.
	\label{lem:moz}
\end{lemma}
\begin{proof}
The uniqueness of $\t_k$ is clear, since $\graph{f}$ is a tree.  The other assertions follow  from the fact that the rotators of  the left- and rightmost vertices of $\arch{f}$ are increasing.
\end{proof}

\begin{proposition}
\label{thm:breakdown}
$\hat{F}_n(q,t) = t \qt{n} F_{n-1}(q,t)$.
\end{proposition}

\newcommand{\conj}[1]{\bar{#1}}

\begin{proof}
For $k \in [1,n]$, let $\hatfset_{n,k}$ be the subset of $\F_n$ containing factorizations whose $k$-th factor is $(0\,n)$. Lemma~\ref{lem:moz} shows $\{\hatfset_{n,1},\ldots,\hatfset_{n,n}\}$ is a partition of $\hatfset_n$.

Define the function $\phi_k$ on $\hatfset_{n,k}$ by
\[
	\phi_k(\t_1,\ldots,\t_n) = (\conj{\t}_{k+1},\ldots,\conj{\t}_n,\t_1,\ldots,\t_{k-1}),
\]
where $\conj{\t} =  \s_n \t \s_n^{-1}$ denotes the conjugate of $\t$ under $\s_n$.  
The effect of $\phi_k$ is to ``rotate'' a factorization $f \in \hatfset_{n,k}$ until  $(0\,n)$ is the rightmost factor, at which point this factor is removed.
We claim that $\phi_k$ is bijection from $\hatfset_{n,k}$ to $\F_{n-1}$, and that
\begin{equation}
\begin{aligned}
\label{eq:lowerupper}
	\AL(\fact) &= \AL(\phi_k(\fact))+k-1 \\
	\AU(\fact) &= \AU(\phi_k(\fact))+n-k+1. 
\end{aligned}
\end{equation}

We first verify $\phi_k(\hatfset_{n,k}) \subseteq \F_{n-1}$.  Let $\fact=(\t_1,\ldots,\t_n) \in \hatfset_{n,k}$, so that $\t_1\cdots\t_n = \s_n$ and $\t_k=(0\,n)$. Taking $\a = \t_1 \cdots \t_k$ and $\b = \t_{k+1} \cdots \t_n$, we have $\s_n^{-1} = \b^{-1}\a^{-1}$ and  therefore
\[
	\s_n = \s_n \beta(\beta^{-1}\alpha^{-1})\alpha = (\s_n \b \s_n^{-1})\alpha =\conj{\t}_{k+1} \cdots \conj{\t}_n \t_1 \cdots \t_k.
\] 
Multiplying  on the right by $\t_k=(0\,n)$ gives $\conj{\t}_{k+1} \cdots \conj{\t}_n \t_1 \cdots \t_{k-1}= \s_n(0\,n) = \s_{n-1} $.    Thus $\phi_k(f)$ is a factorization of $\s_{n-1}$ into $n-1$ transpositions; \emph{i.e.} $\phi_k(f) \in \F_{n-1}$. 

The argument above is easily reversed to confirm the inverse of $\phi_k$ is given by 
\[
\phi_k^{-1}(\t_1,\ldots,\t_{n-1}) = (\t_{n-k+1},\ldots,\t_{n-1}, (0\,n),\t_1',\ldots,\t_{n-k}'),
\]
where $\t' =\s_n^{-1} \t \s_n$.  Therefore $\phi_k$ is bijective.

Now suppose the factors of $\fact=(\t_1,\ldots,\t_n) \in \hatfset_{n,k}$ are $\t_i = (a_i\,b_i)$, with $a_i < b_i$.    Lemma~\ref{lem:moz} ensures $a_i > 0$  for $i > k$, whence $\conj{\t_i}= (a_i-1\;b_i-1)$. That is,
\begin{equation}
	\label{eq:rotate}
  \begin{aligned}
	L(\phi_k(\fact)) &= (a_{k+1}-1,\ldots,a_n-1,a_1,\ldots,a_{k-1}) \\
	U(\phi_k(\fact)) &= (b_{k+1}-1,\ldots,b_n-1,b_1,\ldots,b_{k-1}).
  \end{aligned}
  \end{equation}
Since $a_k=0$, we have
  \begin{align*}
\AL(\phi_k(\fact)) &= \binom{n-1}{2}-\Big(\sum_{i=1}^n a_i-(n-k)\Big) \\
	&=\binom{n}{2}-\sum_{i=1}^n a_i -k + 1  \\
	&= \AL(\fact)-k+1.
  \end{align*} 
Similarly, since $b_k=n$,
  \begin{equation*}
	\AU(\phi_k(\fact)) = \Big(\sum_{i=1}^n b_i - (n-k)-n\Big) - \binom{n-1}{2} 
		 = \AU(\fact)-n+k-1.
  \end{equation*}
This establishes~\eqref{eq:lowerupper}, from which there follows
  \begin{equation*}
	\sum_{\fact \in \hatfset_{n,k}} t^{\AU(\fact)} q^{\AL(\fact)}
	= q^{k-1} t^{n-k+1} \sum_{\fact' \in \F_{n-1}} t^{\AU(\fact')} q^{\AL(\fact')}.
  \end{equation*}
Summing  over $k \in [1,n]$ completes the proof.
\end{proof}

Propositions~\ref{thm:exponential} and~\ref{thm:breakdown} together complete the proof of Theorem~\ref{thm:mainthm}.  We note it is straightforward to use the arguments in this section to define a recursive bijection between $\F_n$ and $\T_n$ preserving the relevant statistics.

\section{Special Cases and Related Results}
\label{sec:special}

An interesting special case is the terms of highest degree in
$F_n(q,t)$;  that is, the factorizations with maximum total difference.
\newcommand{\fmax}{\F^\textnormal{max}}
Let $\fmax_n \subseteq \F_n$ be the set of factorizations with maximum total difference.
While from the point of view of factorizations it is \emph{a priori}
unclear what this maximum total difference is and what factorizations have it, from the point of view of tree inversions this is clear.  The terms
of maximum degree in $I_n(q,t)$ correspond to trees with maximum total
depth.  Such trees are clearly paths, and therefore correspond to permutations, and have
total depth $1 + 2 \cdots + n =
\binom{n+1}{2}$.  Furthermore, the
notions tree inversions/coinversions for paths correspond to the usual notions of
inversions/coinversions for permutations.   Thus, the terms of maximum degree in
$I_n(q,t)$ are in fact given by the extensively studied \emph{inversion polynomial
for permutations} (see Stanley \cite[Section 1.3]{stanenum1}).  This gives us the following
corollary.
\begin{corollary}\label{cor:perminv}
	\begin{equation*}
		\sum_{f \in \fmax_n} q^{\AL(f)} t^{\AU(f)} = t^n \prod_{i=1}^{n} \frac{t^i - q^i}{t-q}.
	\end{equation*}
\end{corollary}

\newcommand{\rarrowf}{\overrightarrow{\F}}
\newcommand{\larrowf}{\overleftarrow{\F}}
\newcommand{\rarrowfs}{\overrightarrow{F}}
\newcommand{\larrowfs}{\overleftarrow{F}}

Other special cases can obtained by looking at factorizations in
$\F_{n}$ with restrictions placed on
$L(\fact)$.
\newcommand{\fperm}{\F_n^{\mathrm{perm}}}
For example, let $\fperm$ be the subset of $\fact \in \F_n$ where $L(\fact)$ is a
permutation, and construct the series
\begin{equation*}
	F_n^\mathrm{perm}(q,t) = \sum_{\fact \in \fperm} t^{\AU(\fact)}
	q^{\AL(\fact)}.
\end{equation*}
It is clear that any $\fact \in \F_n$ has $\AL(\fact) = 0$ if and
only if $L(\fact)$ is a permutation.  It follows that
$F_n^\mathrm{perm}(q,t)$ is a polynomial in $t$ and is equal
to $F_n(0,t) = I_n(0,t)$.  The latter polynomial is evidently the total
depth polynomial for \emph{increasing trees};  that is, trees with no
inversions.

We can, of course, restrict our attention to other sets.  An \emph{increasing} parking function $(a_1, \dots, a_n)$ is a parking function
with $a_i \leq a_{i+1}$ for all $i \geq 1$.  \emph{Decreasing} parking functions
are analogously defined.  Both sets of parking functions are enumerated
by the Catalan numbers, as they both have clear interpretations as Dyck paths.
However, when these sets of parking functions are viewed as the lower sequences of
factorizations, the upper sequence has some interesting properties.  We call
factorizations in $\F_n$ with increasing and decreasing lower sequences \emph{increasing} and
\emph{decreasing factorizations} of $\sigma_n$, respectively, and denote their respective sets by
$\rarrowf_n$ and $\larrowf_n$.  Decreasing factorizations of the
canonical full cycle have been studied
previously in \cite{gewurzmerola}.  We note that the 
authors there regard their factorizations as increasing since they multiply
permutations from right to left.  We will use and revisit their
results in the proof of Theorem \ref{thm:notfancyqt}.

For $n \geq 1$ define the series $C_n(q,t)$ by
\begin{equation}\label{eq:adinroich}
	C_{n}(q,t) := \sum_{k=0}^{n-1} q^k t^{n-k-1} C_k(q,t)
	C_{n-k-1}(q,t),\; \text{with $C_0(q,t) = 1$.}
\end{equation}
The polynomials $C_n(q,t)$ are natural generalizations of Catalan
numbers, and were recently considered in a similar context by Adin and Roichman
\cite{adinroich}, where the authors are determining the radius of the
\emph{Hurwitz graph} of the symmetric groups.  
Finally, let $\rarrowfs_n(q,t) = \sum_{\fact \in \rarrowf_n} q^{\AL(\fact)}
t^{\AU(\fact)}$ and
$\larrowfs_n(q,t) = \sum_{\fact \in \larrowf_n} q^{\AL(\fact)} t^{\AU(\fact)}$

\begin{theorem} \label{thm:notfancyqt}
	For all $n \geq 1$,  the series $\rarrowfs_n(q,t)$ satisfies
	\begin{equation*}
		\rarrowfs_{n}(q,t) = \sum_{k=0}^{n-1} q^k t^{n-k}
		\rarrowfs_k(q,t) \rarrowfs_{n-k-1}(q,t), \mathrm{ with }\;
		\rarrowfs_0(q,t) = 1,
	\end{equation*}
	whence, $\rarrowfs_{n}(q,t) = t^n C_n(q,t)$.  Also, for all $n \geq 0$,
	$\larrowfs_n(q,t)=t^n C_n(q,1)$.
\end{theorem}
\newcommand{\avd}{\mathrm{a}_{132}}

\begin{proof}
	We first observe that any factorization $\fact \in
	\rarrowf_n$ contains $(0\; n)$ as a factor, and we refer the reader to
	Figure \ref{fig:decomposition} for an explanatory illustration of this.  Indeed, if
the arch diagram of a factorization $f \in \rarrowf_n$ has more than one cap (as
	defined in Section \ref{sec:mainthm}), then the cap labels 
	will be increasing from right to left.  Recall that $L(f)_i$,
	the $i$-th entry of $L(f)$,  is given by the left endpoint of
	the arch labelled $i$.  Hence, if $i$ and $j$
	are labels on two adjacent caps with $i$ to the left of $j$, then
	$i > j$ but $L(f)_i < L(f)_j$, contradicting that $f$ has an 
	increasing lower sequence.  We conclude therefore that the arch
	diagram of $f$ has only one cap, $(0\; n)$ is a factor of
	$f$,  and by Lemma \ref{lem:moz} this factor is the rightmost factor containing the
	symbol $0$.   Notice further that since $f$ is increasing, all factors containing 0
	are at the beginning of $f$.   Furthermore, these factors
	containing 0 appear in increasing order of their second
	element because of the increasing rotator condition on arch
	diagrams.

	Now suppose that $\fact \in \rarrowf_n$ and  let $A = \arch{f}$.  Removing the cap from $A$ leaves two arch diagrams,  $A_1$ and $A_2$, corresponding
	to (minimal) increasing factorizations $\fact_1$ and $\fact_2$ of $(0\; 1\; \cdots k)$ 
	and $(k+1\; \cdots n)$, respectively.  The reasoning in the
	above paragraph applies to $f_1$ and $A_1$ as well:  $A_1$ has a
	unique cap, it corresponds to a transposition $(0\; k)$ in
	$f_1$, and this transposition is the rightmost transposition of $f_1$
	containing 0.  Note that the symbol $k$ is also the second
	largest symbol next to $n$ such that $(0\; k)$ is a factor
	of $f$.
	Because $\fact$ is an increasing factorization, all the factors
	of $\fact_1$ occur to the left of the factors of $\fact_2$
	
	We therefore see that $\fact_1 \in
	\rarrowf_k$, but $\fact_2 \notin \rarrowf_{n-k-1}$; however, it is clear that
	the factorization $\fact'_2$, which is obtained from $\fact_2$ by
	subtracting $k+1$ from every element in every factor, is a minimal increasing factorization of $(0\;
	\cdots\; n-k-1)$, and so $\fact'_2 \in \rarrowf_{n-k-1}$.  
	
	The above sets up a map from $\rarrowf_n$ to $\bigcup_{k=0}^{n-1} \rarrowf_k \times
	\rarrowf_{n-k-1}$.  It is easy to see that the above construction is
	reversible.  Namely, given an $\fact_1 \in \rarrowf_k$ and $\fact_2' \in
	\rarrowf_{n-k-1}$, we can reconstruct $\fact_2$ from the length of $\fact_1$,
	and then $\fact$ is the concatenation of $\fact_1$ and $\fact_2$ with the transposition $(0\; n)$ inserted
	immediately after the rightmost factor in the concatenation with a 0.  We therefore
	have a bijection from $\rarrowf_n$ to $\bigcup_{k=0}^{n-1} \rarrowf_k \times
	\rarrowf_{n-k-1}$.  What remains is verifying that areas are preserved.
	To wit,
	\begin{align*}
		\AU(\fact) &= \sum_{i=1}^n U(\fact)_i - {n \choose 2}\\
		&= \left(\sum_{i=1}^k U(\fact_1)_i + \sum_{i=1}^{n-k-1}
		\left(U(\fact'_2)_i +
		k+1\right) + n\right) - {k \choose 2}  - \sum_{i=k}^{n-1} i \\
		&= n-k + \AU(\fact_1) + \sum_{i=1}^{n-k-1} U(\fact'_2)_i - {n-k-1 \choose 2}\\
		&= n-k + \AU(\fact_1) + \AU(\fact'_2).
	\end{align*}
	Similarly, we can show that
	\begin{equation*}
		\AL(\fact) = k + \AL(\fact_1) + \AL(\fact'_2).
	\end{equation*}
	This completes the proof that the polynomials
	$\{\rarrowfs_n(q,t)\}$ satisfy the
	recurrence stated in the theorem.   By \eqref{eq:adinroich},
	the polynomials $\{t^n C_n(q,t)\}$ satisfy the same recurrence and
	initial conditions, and so for all $n \geq 0$, $\rarrowfs_n(q,t) =
	t^nC_n(q,t)$. 
	
	The two claims about $\larrowfs_n(q,t)$ in Theorem \ref{thm:notfancyqt}
	are less novel and in fact follow easily from two previously known results. Indeed, it was shown by Gewurz and Merola
	\cite{gewurzmerola} that $U : \larrowf_n \rightarrow \avd(n)$ is a
	bijection, where $\avd(n)$ is the set of $132$-avoiding permutations in
	the symmetric group acting on the symbols $\{1, 2, \dots, n\}$.
	It is well-known that $\avd(n)$ is counted by Catalan numbers.  Thus, the upper sequences of
	 factorizations in $\larrowf_n$ are permutations and therefore each $f \in
	\larrowf_n$ has $\AU(f) =
	\sum_{i=1}^n i - \binom{n}{2} = n$.  Whence the contribution of $f$
	to $\larrowfs_{n}(q,t)$ is $t^n q^i$, where $i$ is the lower area of $f$.
	Finally, as noted above, decreasing parking functions are in fact Dyck paths,
	and the classic result of Carlitz  and Riordan \cite{carlitzriordan} gives the
	area polynomial of Dyck paths of length $n$ as $C_n(q,1)$.
	The result then follows.	
\end{proof}

\section{Factorizations and Parking Functions}
\label{sec:unimodal}

The aim of this section is to prove Theorem~\ref{thm:unimodal}. This will be done  through a sequence of propositions, beginning with the following elementary result:

\begin{proposition}\label{prop:alwayspark}
For any $\s\in \cycles{n}$, we have $\Ls(\F_\s)\subseteq \p_n$ and $\Us(\F_\s) \subseteq \M_n$.
\end{proposition}

Our proof of Proposition~\ref{prop:alwayspark} relies on the following well-known characterization of parking functions, which is  itself readily established from their definition:

\begin{lemma}\label{lem:pfcharacterization}
A sequence $ (a_1, \dots, a_n) \in \N^n$ is a parking
function if and only if for each $i \in [1,n]$ it contains at least $i$
entries  less than $i$ (equivalently, at most $i$ entries  are greater than or equal to $n-i$). \qed
\end{lemma}

\begin{proof}[Proof of Proposition~\ref{prop:alwayspark}]
By Lemma~\ref{lem:pfcharacterization}, $L(\fact) \in \p_n$ if and only if $\graph{f}$ has at least $i$ edges with an endpoint less than $i$, for each $i \in [1,n]$.  To see that this is so, let $G_i$ be the subgraph of $G(f)$ induced by vertices $[i,n]$. Since $\graph{f}$ is a tree (by Lemma~\ref{lem:minimal1}), $G_i$ is a forest with $n-i+1$ vertices and hence at most $n-i$ edges.  Thus $\graph{f} \setminus G_i$ has at least $i$ edges, as desired.  The proof that $U(\fact) \in \M_n$ is similar.
\end{proof}

We now show  that  unimodality of $\s$ ensures surjectivity of $\map{\Ls}{\F_\s}{\p_n}$. This is the crux of  Theorem~\ref{thm:unimodal}.  As proof we give an algorithm that explicitly constructs an element of $\Ls^{-1}(p)$.

The algorithm is most simply described in graphical terms.  Given a parking function $p=(a_1,\ldots,a_n)$ and unimodal cycle $\s=(0\,s_1\, \cdots \, s_n)$, we build an arch diagram as follows, being sure to abide by the  embedding rules throughout  (\emph{i.e.} edges are drawn above the axis without crossings and rotators remain increasing).
 Figure~\ref{fig:biane1} illustrates the process.   
\begin{figure}[t]
	\centering
	\includegraphics[width=.9\textwidth]{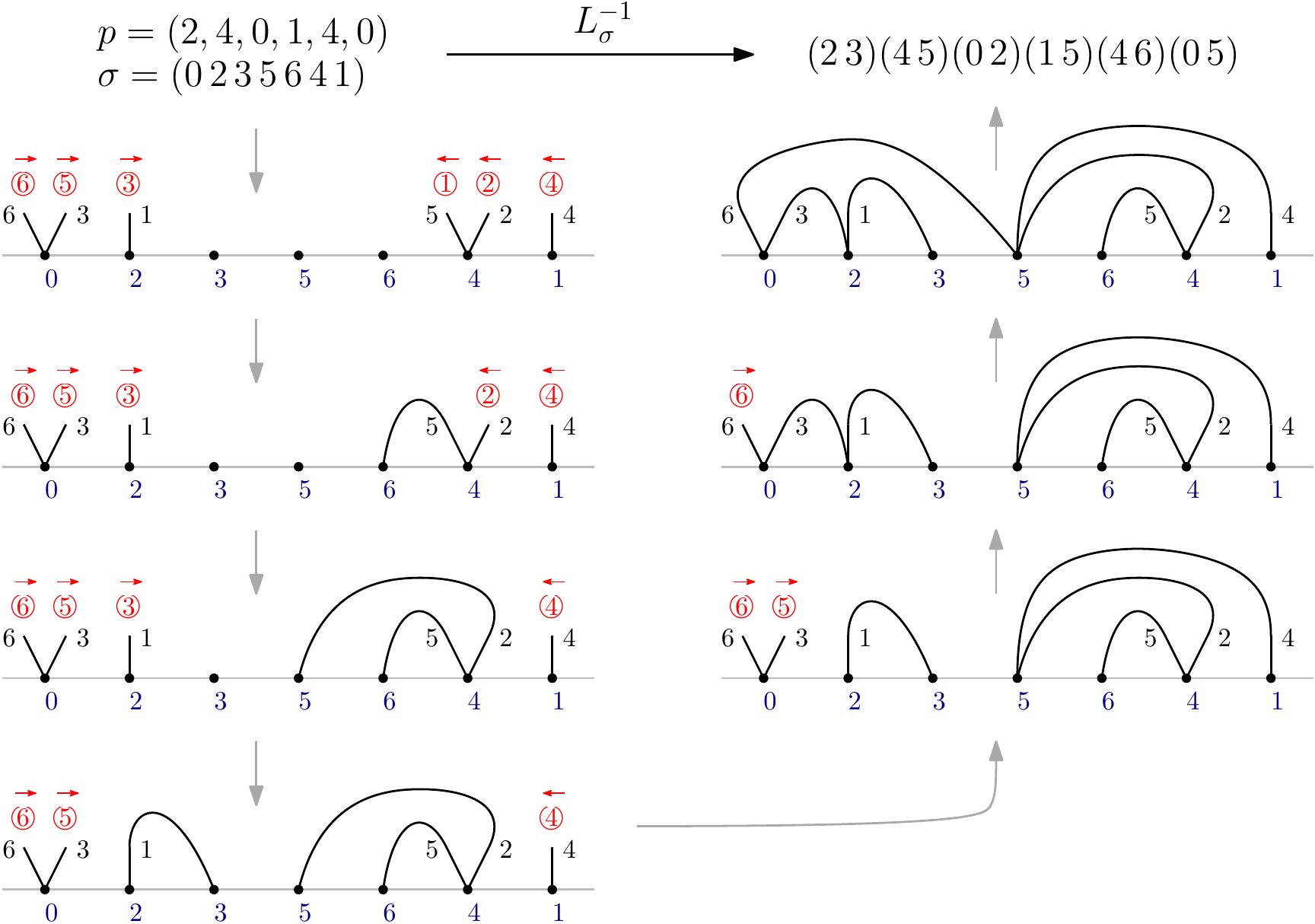}
	\caption{Constructing $\Ls^{-1}(p)$ for 
	$p=(2,4,0,1,4,0) \in \p_6$ and $\s = (0\,2\,3\,5\,6\,4\,1) \in \cycles{6}$.  Circled numbers indicate the order in which half-edges are processed, and the overlying arrows indicate the direction in which they are to be extended to form arcs. }
	\label{fig:biane1}
\end{figure}
\begin{itemize}
\item	{\scshape Initialize:}
\begin{itemize}
\item	Arrange vertices $0, s_1, \ldots,s_n$ in left-to-right order along the $x$-axis.
\item  	For each  $j \in [1,n]$, attach  to vertex $a_j$ a half-edge with label $j$.
\end{itemize}

\item	{\scshape Iterate:}
\begin{itemize}
\item	Let $v$ be the largest vertex with at least one incident half-edge. If no such $v$ exists,  the algorithm terminates.
\item	If $v$ is left of vertex $n$,  extend its \emph{least} incident half-edge to connect it with the first possible vertex to its \emph{right}.
\item	If $v$ is right of vertex $n$,  extend its \emph{greatest} incident half-edge to connect it with the first possible vertex to its \emph{left}.
\end{itemize}
\end{itemize}

We claim that the above procedure can always be followed to completion, and that the resulting arch diagram $A$ satisfies $\Ls(\FACTs{\s}{A})=p$.  To prove the claim we shall restate the algorithm more formally in terms of factorizations. Some additional terminology will be convenient.

\newcommand{\operm}[2]{\omega_{#1,#2}}

Let $\s = (s_0\,\,s_1\,\cdots\,s_n) \in \cycles{n}$ be unimodal, with $s_0=0$ and $s_N=n$.  We say an index $i \in [n-1]$ is \emph{$\s$-left}  (resp. \emph{$\s$-right}) if $i = s_k$ for some $k < N$ (resp. $k > N$).
Now fix $p=(a_1,\ldots,a_n) \in \p_n$ and let $J_i$ be the list of all  indices $j$ such that $a_j=i$, taken in increasing order when $i$ is $\s$-left and in decreasing order when $i$ is $\s$-right.  Let $\operm{\s}{p}$ denote the permutation of $[1,n]$ defined by the concatenation $J_{n-1}, J_{n-2}, \ldots,
J_0$.  

With respect to the  graphical algorithm, $\s$-left (resp. $\s$-right)  elements correspond to vertices whose half-edges are extended to the right (left) to form
complete arcs, while $\operm{\s}{p}$ describes the order in which half-edges are processed.  Observe that
the sequence $(a_{\operm{\s}{p}(i)})_{1 \leq i \leq n}$ must be weakly decreasing.  

Finally, we declare a permutation $\pi \in \Sym{[n]}$ to be \emph{$\s$-contiguous} if each of its cycles is of the form $(s_i\,s_{i+1}\,\cdots\,s_{j})$ for some $i \leq j$,

\begin{example}\label{examp:algo1}  
	We continue with the same inputs as in Figure
	\ref{fig:biane1}, so set $\s=(0\,2\,3\, 5\, 6\, 4\,1) \in \cycles{6}$.
 The labels $\{0,2,3,5\}$ are $\s$-left and $\{1,4\}$ are $\s$-right.  For  $p=(2,4,0,1,4,0) \in \p_6$ we have $J_5 = \emptyset, J_4 = (5,2), J_3 =
\emptyset, J_2 = (1), J_1=(4)$ and $J_0 = (3,6)$, so $\operm{\s}{p} =
(5,2,1,4,3,6)$ and $(a_{\operm{\s}{p}(i)})_{1 \leq i \leq 6} = (4,4,2,1,0,0)$.   
Both $(0\,2)(5\,6\,4)$ and $(6\,4)$ are $\s$-contiguous, whereas $(2\,3\,4)(5\,6)$ is not.
\end{example}

\begin{algorithm}
Let $\s \in \cycles{n}$ be unimodal and let $p=(a_1,\ldots,a_n) \in \p_n$. The
following procedure terminates with a factorization $f = (\t_1,\ldots,\t_n) \in
\F_{\s}$ such that $\Ls(f)=p$.  
\begin{tabbing}
mm \= mmm\= mm\= mm\= mm\= mm\= mm\= \kill
\> 01 \> $(\t_1,\ldots,\t_n) \gets (\pid,\ldots,\pid)$\\
\> 02 \> $\pi_1 \gets \pid$ \\
\> 03 \> {\bf for} $i$ {\bf from} 1 {\bf to} $n$  {\bf do}\\
\> 04 \> \> $j \gets \operm{\s}{p}(i)$ \\
\> 05 \> \> {\bf if} $a_j$ is $\s$-left  {\bf then} \\
\> 06 \> \> \> $b_j \gets \pi_{i}^{-1}\s\t_n \t_{n-1} \cdots \t_{j+1}(a_j)$ \\
\> 07 \> \> {\bf else} \\
\> 08 \> \> \> $b_j \gets \pi_i \s^{-1}\t_1 \t_2 \cdots \t_{j-1}(a_j)$ \\
\> 09 \> \> {\bf end if} \\
\> 10 \> \> $\t_j \gets (a_j\,b_j)$ \\
\> 11 \> \> $\pi_{i+1} \gets \t_1 \cdots \t_n$ \\
\> 12 \> {\bf end for}
\end{tabbing}
\label{algo:unimodal}
\end{algorithm}

Algorithm~\ref{algo:unimodal} builds $\Ls^{-1}(p)$ through a sequence of ``partial factorizations'' of $\s$ --- that is,   tuples $(\t_1,\ldots,\t_n)$ such that each $\t_j$ is either  a transposition or the identity   and the product $\t_1 \cdots
\t_n$ is $\s$-contiguous.   The initial state is  $(\t_1,\ldots,\t_n)=(\iota,\ldots.\iota)$, and each iteration  replaces a distinct factor $\t_j$ with a transposition of the form $(a_j\,b_j)$.  After $n$ iterations this results in a partial factorization of $\s$  composed of $n$ transpositions, which is necessarily an element of $\F_\s$ with lower sequence $p$.

Table~\ref{tab:algexample} shows the algorithm being applied to the
same input as in the prior graphical example in Figure~\ref{fig:biane1}.  This allows for comparison of the two approaches and should help guide the reader through the following proof of correctness.  Note that the table also displays the value of certain parameters arising in the proof.

\newcommand{\trans}[2]{(#1\,#2)}
\newcommand{\highlight}[1]{#1}
\newcolumntype{L}{>{$}c<{$}} %
\newcommand{\mcyc}[1]{{\color{red} (#1)}}
\newcommand{\Mcyc}[1]{{\color{blue} (#1)}}
\begin{table}[t]
\caption{Algorithm~\ref{algo:unimodal} applied to $\s = (0\,2\,3\,5\,6\,4\,1) \in \cycles{6}$ and $p=(2,4,0,1,4,0) \in \p_6$, with  $\operm{\s}{p}=(5,2,1,4,3,6)$. Column  $s$ displays $s_{\ell+1}$ when $a_j$ is $\s$-left and $s_{\ell-1}$ when $a_j$ is $\s$-right.  Cycles $C$ and $C'$ are highlighted in red and blue, respectively. } 

\centering

\begin{tabular}{L|LLLLLL|L|L|L|L|L}
i &\t_1 & \t_2 & \t_3 & \t_4 & \t_5 & \t_6 & \pi_i & j & a_j & s  & b_j \\ \hline
1 &  
\pid & \pid & \pid & \pid & \pid  & \pid &
(0)(1)(2)(3)\mcyc{4}(5)\Mcyc{6} &
5 &
4 &
6 & 6\\
2 & \pid &  \pid & \pid & \pid & \highlight{\trans{4}{6}} & \pid & (0)(1)(2)(3)\mcyc{4\,6}\Mcyc{5}  &
2  &
4 &
5 & 5\\
3 & \pid & \trans{4}{5} & \pid & \pid & \highlight{\trans{4}{6}} & \pid &
(0)(1)\mcyc{2}\Mcyc{3}(4\,5\,6) & 
1 & 
2 &
3 & 3\\
4 & \highlight{\trans{2}{3}} & \trans{4}{5} & \pid & \pid & \trans{4}{6} & \pid &
 (0)\mcyc{1}(2\,3)\Mcyc{4\,5\,6} & 
4 &
1 &
4 & 5\\
5 & \trans{2}{3} & \trans{4}{5} & \pid & \highlight{\trans{1}{5}} & \trans{4}{6} & \pid &
\mcyc{0}(1\,5\,6\,4)\Mcyc{2\,3} &
3 &
0 &
2 & 2\\
6 & \trans{2}{3} & \trans{4}{5} & \highlight{\trans{0}{2}} & \trans{1}{5} & \trans{4}{6} & \pid  &
\mcyc{0\,2\,3}\Mcyc{1\,5\,6\,4} &
6 &
0 &
5 & 5 \\
7 & \trans{2}{3} & \trans{4}{5} & \trans{0}{2} & \trans{1}{2} & \trans{4}{6} & \highlight{\trans{0}{5}} &
(0\,2\,3\,5\,6\,4\,1)  & & &
\end{tabular}
\label{tab:algexample}
\end{table}

\begin{proof}[Proof of correctness:]
Say $\s=(s_0\,\ldots,s_n)$, where $s_0=0$ and $s_N = n$, and for brevity let $w_i=\operm{\s}{p}(i)$ for $i \in [n-1]$. Suppose  the following conditions are met upon entering the $i$-th iteration of the loop, as is clearly the case when $i=1$:
\begin{itemize}
\item[(A)]	$\pi_i = \t_1\cdots\t_n$ is $\s$-contiguous and  $\len{\pi_i}=n+2-i$  

\item[(B)] If $r \in \{w_1,\ldots,w_{i-1}\}$ then $\t_r = (a_r\,b_r)$  for some $b_r > a_r$, otherwise $\t_r = \pid$.
\end{itemize}
We prove the result inductively by verifying that these conditions continue to hold, with $i$ replaced by $i+1$, upon completion of the $i$-th iteration. (A)  then implies the algorithm terminates with a factorization $f=(\t_1,\ldots,\t_r)$ of some $\s$-contiguous full cycle, which much of course be  $\s$ itself, while (B) ensures that $\Ls(f)=p$.

\newcommand{\lessthan}{\textcircled{\scriptsize 1}}
\newcommand{\contained}{\textcircled{\scriptsize 2}}

Let $j = w_i$ and let $C$ be the cycle of $\pi_i$ containing $a_j$.  We shall  make repeated use of the following assertions:
\begin{enumerate}
\item[\lessthan{}]	$\pi_i$ fixes all elements $< a_j$
 \item[\contained{}] $\supp{C}$ does not contain the entire interval $[a_j,n]$
\end{enumerate}
The \emph{support} of a cycle $C=(c_1\,\cdots\,c_k)$ is defined as $\supp{C}:=\{c_1,\ldots,c_k\}$. 

Clearly \lessthan{} follows from hypothesis (B) and the fact that   $(a_{w_1},\ldots,a_{w_n})$ is weakly decreasing.  To prove \contained{}, note that Lemma~\ref{lem:pfcharacterization} gives $a_j \leq n-i$. Thus $[a_j,n]$ contains at least $i+1$ elements. But $C$ is a cycle of $\pi_i$, and $\pi_i$ is a product of $i-1$ transpositions, so $C$ is of length at most $i$. 

The remainder of the proof is elementary but technical.  
The schematic in Figure~\ref{fig:schematic} illustrates the equivalence with the graphical version of the algorithm.
\begin{figure}[t]
	\centering
	\includegraphics[width=.8\textwidth]{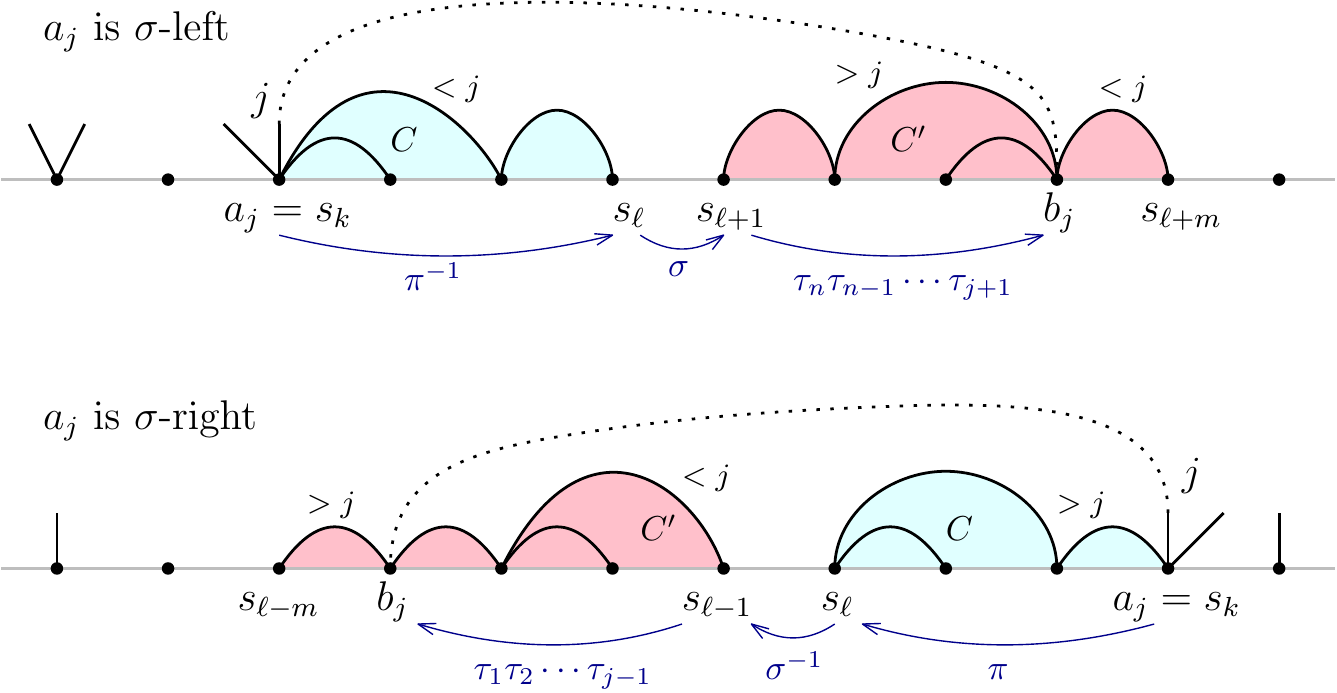}
\caption{Schematic for the proof of Proposition~\ref{prop:unimodal}.}
	\label{fig:schematic}
\end{figure}

\medskip
\noindent{\emph{Case 1:} $a_j$ is $\s$-left}
\medskip

Say $a_j=s_k$, where $k < N$. We claim  $C = (s_k\, s_{k+1}\,\cdots\,s_{\ell})$
for some $\ell \in [k,n-1]$.  As proof, observe that $\pi_i$ being $\s$-contiguous means either $C = (s_k\, s_{k+1}\, \cdots\, s_{\ell})$ for some $\ell \geq k$ or $s_{k-1} \in \supp{C}$.  The latter is impossible, as unimodality of $\s$ gives $s_{k-1} < s_k$ for $k < N$ while \lessthan{} implies $s_k$ is the least element of $C$.  And we cannot have $C=(s_k\,\cdots\,s_n)$, as unimodality would then give  $\supp{C} \supseteq [s_k,n]$, contrary to \contained{}.

Let $C'$ be the cycle of $\pi_i$ containing $s_{\ell+1}$. Since $\pi_i$ is $\s$-contiguous, we have $C'=(s_{\ell+1} \,\cdots\, s_{\ell+m})$ for some $m \geq 1$. Then $C$ and $C'$ are merged into the single cycle $(s_k \,s_{k+1}\, \cdots\, s_{\ell+m})$ of the permutation $\rho=\pi_i(s_k\,s_{\ell+1})$.  Clearly $\rho$ is $\s$-contiguous and composed of $n+1-i$ cycles. 

Now let $\a = \t_1 \cdots \t_{j-1}$ and $\b = \t_{j+1} \cdots \t_n$. Then $\pi_i =
\a\b$, as $\t_j = \pid$, and
\begin{equation}
\label{eq:joined}
\rho = \pi_i(s_k\,s_{\ell+1})
=
	\a\big( \b(s_k\,s_{\ell+1})\b^{-1}) \b 
	= 
	\a\big( \b^{-1}(a_j)\; \b^{-1}(s_{\ell+1})\big)\b
\end{equation}
We claim $\b^{-1}(a_j)=a_j$  and $\b^{-1}(s_{\ell+1}) > a_j$.  By \lessthan{}, no $\t_r$ can move any symbol smaller than $a_j$.  So if $\t_r$ moves $a_j$ then (B) gives $r \in \{w_1,\ldots,w_{i-1}\}$ and $a_r = a_j$, whence $r \leq j$ by definition of $\operm{\s}{p}$. Thus $\t_r(a_j)=a_j$ for all $r > j$, giving $\b^{-1}(a_j)=a_j$. Now observe that $s_{\ell+1} > s_k=a_j$, as otherwise unimodality would imply $[s_k,n] \subseteq \{s_k, s_{k+1},\cdots, s_{\ell}\} = \supp{C}$, contradicting \contained{}.  Since none of the $\t_r$ move symbols smaller than $a_j$, we conclude $\b^{-1}(s_{\ell+1}) > a_j$.  

Finally, let $b_j = \b^{-1}(s_{\ell+1})=\pi_i^{-1}\s\b^{-1}(a_j)$ and set $\t_j=(a_j\,b_j)$ as in lines 06 and 10 of the algorithm.  We have shown $a_j < b_j$, and by~\eqref{eq:joined} we see that line 11  sets $\pi_{i+1} =\t_1 \cdots\t_n = \a(a_j\,b_j)\b = \rho$. Thus $\pi_{i+1}$ is $\s$-contiguous and has  $n+1-i$ cycles.

\medskip
\noindent{\emph{Case 2:} $a_j$ is $\s$-right}
\medskip

The proof is similar to the $\s$-left case and details will be omitted. If $a_j=s_k$, where $k > N$, then we have $C = (s_{\ell}\, \cdots\, s_{k})$ for some $\ell \in [1,k]$.  Let $C'=(s_{\ell-m}\,\cdots\,s_{\ell-1})$ be the cycle of $\pi_i$ containing $\s_{\ell-1}$, and note that $C$ and $C'$ are merged into the cycle $(s_{\ell-m}\,\cdots\,s_k)$ of the $\s$-contiguous permutation $\rho = (s_k\,s_{\ell-1})\pi_i$. With $\a$ and $\b$ as above, we arrive at $\pi_{i+1} = \rho =  \a(a_j\,b_j)\b$, where $b_j = \a(s_{\ell-1}) > a_j$. 
\end{proof}

Notice that Algorithm~\ref{algo:unimodal} simplifies considerably in the case $\s=\s_n$, where every $i \in [n-1]$ is $\s$-left.  Suppose, in the $i$-th iteration, that $C$ is the cycle of
$\pi_i$ containing $a_j$. Let $m$ be  the largest element of $C$ and let $C'$ be the cycle
of $\pi_i$ containing $m+1$. Then  $b_j$ is the image of $m+1$ under $\t_n
\t_{n-1}\cdots\t_{j+1}$, and $\pi_{i+1}$ is obtained from $\pi_i$ by merging cycles $C$
and $C'$.  Table~\ref{tab:algexample2} provides an example.

\begin{table}[t]
\caption{Algorithm~\ref{algo:unimodal} applied to $\s = (0\,1\,2\,3\,4\,5\,6) \in \cycles{6}$ and $p=(2,4,0,1,4,0) \in \p_6$, with $\operm{\s}{p} = (5,2,1,4,3,6)$. The cycles $C$ and $C'$ are highlighted in red and blue, respectively.} 
\centering
\begin{tabular}{L|LLLLLL|L|L|L|L|L}
i &\t_1 & \t_2 & \t_3 & \t_4 & \t_5 & \t_6 & \pi_i & j & a_j & m & b_j \\ \hline
1 &  
\pid & \pid & \pid & \pid & \pid  & \pid &
(0)(1)(2)(3)\mcyc{4}\Mcyc{5}(6) &
2 &
4 &
4 & 5\\
2 & \pid & \highlight{\trans{4}{5}}   & \pid & \pid & \pid & \pid &
(0)(1)(2)(3)\mcyc{4\,5}\Mcyc{6}  &
5  &
4 &
5 & 6\\
3 & \pid & \trans{4}{5} & \pid & \pid & \highlight{\trans{4}{6}} & \pid &
(0)(1)\mcyc{2}\Mcyc{3}(4\,5\,6) & 
1 & 
2 &
2 & 3\\
4 & \highlight{\trans{2}{3}} & \trans{4}{5} & \pid & \pid & \trans{4}{6} & \pid &
 (0)\mcyc{1}\Mcyc{2\,3}(4\,5\,6) & 
4 &
1 &
1 & 2\\
5 & \trans{2}{3} & \trans{4}{5} & \pid & \highlight{\trans{1}{2}} & \trans{4}{6} & \pid &
\mcyc{0}\Mcyc{1\,2\,3}(4\,5\,6) &
3 &
0 &
0 & 2\\
6 & \trans{2}{3} & \trans{4}{5} & \highlight{\trans{0}{2}} & \trans{1}{2} & \trans{4}{6} & \pid  &
\mcyc{0\,1\,2\,3}\Mcyc{4\,5\,6} &
6 &
0 &
3 & 4 \\
7 & \trans{2}{3} & \trans{4}{5} & \trans{0}{2} & \trans{1}{2} & \trans{4}{6} & \highlight{\trans{0}{4}} &
(0\,1\,2\,3\,4\,5\,6) & & &
\end{tabular}
\label{tab:algexample2}
\end{table}

Taken together with Proposition~\ref{prop:alwayspark}, the following two results complete the proof of Theorem~\ref{thm:unimodal}.

\begin{proposition}
	$\s \in \cycles{n}$ is unimodal if and only if $\map{\Ls}{\F_\s}{\p_n}$ is bijective.
	\label{prop:unimodal}
\end{proposition}

\begin{proof}
It is easy to see that unimodality of $\s$ is necessary for $\Ls$ to be bijective. Indeed, suppose $\s$ is not unimodal, so that  $s_{i-1} > s_i < s_{i+1}$ for some $i$. Then $p=(0, \dots, 0, s_i) \in \p_n$ has two preimages in $\F_\s$ under $\Ls$, namely
$$
(0\,s_1) \cdots \widehat{(0\,s_i)} \cdots (0\,s_n)(s_i\,s_{i+1})
\qquad\text{and}\qquad
(0\,s_1) \cdots \widehat{(0\,s_{i-1})} \cdots (0\,s_n)(s_i\,s_{i-1}),
$$
where the hat indicates removal of the marked transposition.  On the other hand, Algorithm~\ref{algo:unimodal} shows $\Ls$ is surjective  when $\s$ is
unimodal. Bijectivity follows since $|\F_{\s}| = |\p_n| = (n+1)^{n-1}$. 
\end{proof}

\begin{proposition}\label{prop:unimodalu}
$\s \in \cycles{n}$ is unimodal if and only if $\map{\Us}{\F_\s}{\M_n}$ is bijective.
\end{proposition}

\begin{proof}
Define $\switch \in \Sym{[n]}$   by $\switch(i)=n-i$. This involution induces   bijections $\p_n \rightarrow \M_n$  and $\F_{\s}\rightarrow \F_{\switch \s \switch}$ via coordinate-wise action and conjugation of factors, respectively. Note that  $\Us$ is the composition   $\F_{\s} \rightarrow \F_{\switch\s\switch} \overset{L}{\rightarrow} \p_n \rightarrow \M_n$, where $L = L_{\switch\s\switch}$.  Thus $\Us$ is bijective if and only if  $\map{L}{\F_{\switch\s\switch}}{\p_n}$ is  bijective, which by Proposition~\ref{prop:unimodal} is equivalent to $\switch\s\switch$ being unimodal. The result follows since conjugation by $\switch$ clearly preserves unimodality.
\end{proof}

We close this section by reconsidering Algorithm~\ref{algo:unimodal} in the special case $\s = \s_n$.  Note that the algorithm (or its graphical incarnation) can be viewed as the reconstruction of the upper sequence of a factorization $f \in \F_n$ from its lower sequence. Equivalently, it provides a mapping from the  lower path of $f$ to its upper path.  We leave it to the reader to verify that this mapping from $P_L(f)$ and $P_U(f)$ can be succinctly described as follows:  First shift all labels of $P_L$ to the left endpoints of their respective steps. Then, working left-to-right, push each label northeast until it encounters either an unlabelled point on $P_L$ or a point with a smaller label.  The height at which a label of $P_L$ comes to rest is the height at which it occurs in $P_U$.  See Figure~\ref{fig:pushing} and compare with Figure~\ref{fig:factarea}.
\begin{figure}[t]
	\centering
	\includegraphics[width=.8\textwidth]{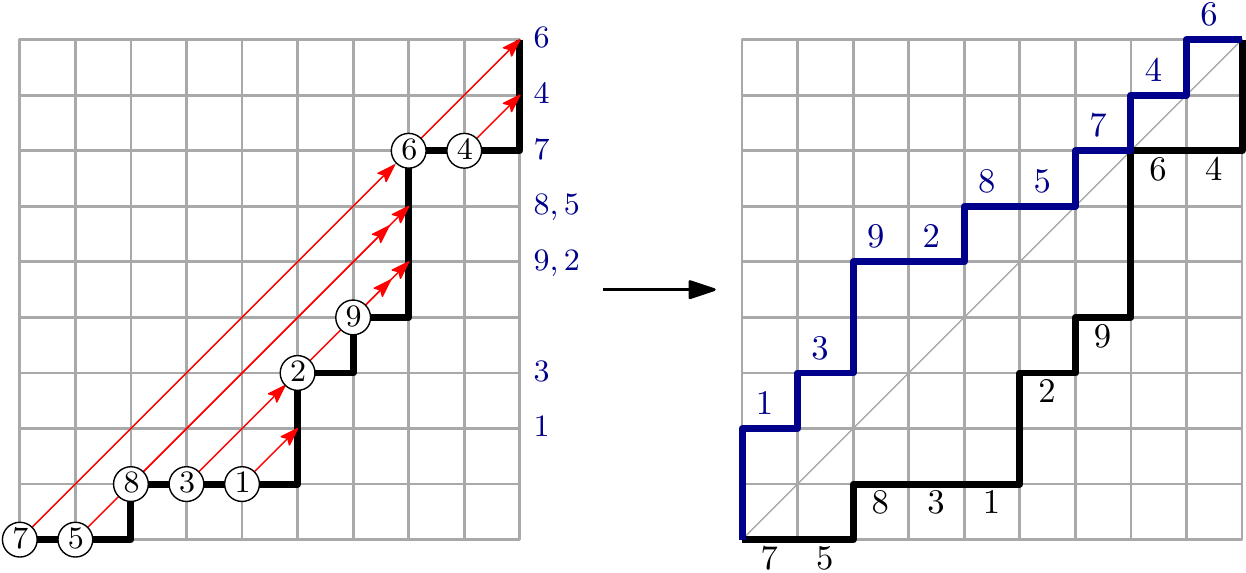}
\caption{Reconstructing the upper path of the factorization $(1\,2)(3\,5)(1\,3)(7\,8)(0\,6)(7\,9)(0\,7)(1\,6)(4\,5) \in \F_9$ from its lower path.}
	\label{fig:pushing}
\end{figure}

\section{The Bounce Statistic}\label{sec:bounce}

Recall the bounce statistic defined in \eqref{eq:defbounce} in Section \ref{sec:ddb}.  Our goal in this section
is to prove a refinement of Theorem \ref{thm:bounce};  namely, we define a
refinement to the bounce statistic  and show that those statistics over
parking functions are equidistributed with the lower and upper area
statistics over factorizations via the connection through tree
inversions.

Consider a parking function $p$ of length $n$ drawn as the usual
labelled Dyck path $P$ on a square grid
between the points $(0,0)$ and $(n,n)$  with its bounce path $B$,
as given in Figure \ref{fig:bouncepath}.  That figure
is repeated in Figure \ref{fig:bounce} (left).   Recall our convention
is to list the labels at the same height in $P$ in decreasing order.  Let $w_1, \ldots,
w_n$ be the permutation of horizontal labels from left to right that occur in $P$, and set $w= w_0, w_1, \dots, w_n$, where $w_0=0$.
For $i \in [n]$, define the sets 
\begin{align*}
	C_{w_i} &:= \textnormal{ the labels of } P \textnormal{ at
	height } i, \textnormal{ and }\\
	D_{w_i} &:=C_{w_i} \cup \bigcup_{j \in C_{w_i}} D_j.
\end{align*}
Because of their recursive definition, it is convenient to find the sets $D_{w_i}$
in the order $i=n, \dots, 0$.  For example, for the
parking function $(1,3,1,7,0,7,0,1,4) \in \p_9$ in Figure
\ref{fig:bounce} (left), we have the permutation $w=(0,7,5,8,3,1,2,9,6,4)$, and the sets $C_i$ and $D_i$ are
given in Table \ref{tab:cidi}.  A useful way to visually display these
ideas is to write
the permutation $w$ from bottom to top to the right of $P$ so $w_i$
is adjacent to the vertical segment with endpoints $(n, i-1)$ and $(n, i)$; then the
labels at height $i$ are seen directly to the left of $w_i$, giving $C_{w_i}$.   See Figure
\ref{fig:bounce} (left).  
\newcolumntype{Y}{>{\centering\arraybackslash}X}
\newcolumntype{E}{>{\,\,} c <{\,\,}}
\begin{table}
	\centering
	\caption{The set $C_i$ and $D_i$ for the parking function in Figure
		\ref{fig:bounce}.  Here, $N = \{1, \dots, 9\}$.}
	\begin{tabular}{|E|E|E|E|E|E|E|E|E|E|E|}
		\hline
		$\textnormal{set} \backslash i$ & $4$ & $6$ & $9$ & $2$ & $1$ &
		$3$ & $8$ & $5$ & $7$ & $0$\\
		\hline
		$C_{i}$ & $\emptyset$ & $\emptyset$ & $4,6$ &  $\emptyset$ &
		$\emptyset$ & $9$ & $2$ & $\emptyset$ & $1,3,8$ & $5,7$\\
		$D_i$ & $\emptyset$ & $\emptyset$ & $4,6$ & $\emptyset$ &
		$\emptyset$ & $4,6,9$ & $2$ & $\emptyset$ & $N/\{5,7\}$ & 
		$N$\\
		\hline
	\end{tabular}
		\label{tab:cidi}
\end{table}

It is useful to regard  $w$ as labels on the vertical steps of the bounce path
$B$ by simply projecting them from the right of the path $P$ to
the vertical steps of $B$.  Define a
\emph{vertical run} of $B$ to be a set of vertical steps of at the same
$x$-coordinate.   Consider the set of labels $V$ on the vertical run
of $B$ with $x$-coordinate $v$.  Then it is straighforward to show:
\begin{itemize}
	\item $D_0 = [1,n]$.
	\item  The sets $\{D_i : i \in V\}$ are pairwise disjoint.
	\item  The set $\bigcup_{i \in V} D_i$ is the set of labels on horizontal steps to the
		right (at any height) of the vertical run
		at $x=v$; so, the total size of this union is $n-v$ (see
		Figure \ref{fig:bounce} (left)).
	\item Whence it follows from the definition in
		\eqref{eq:defbounce} that 
		\begin{equation}\label{eq:bouncedi}
		\bnce(p) = \sum_{i=0}^n |D_i|.
	\end{equation}
\end{itemize}
It further follows from the above properties that for all $i$, $i
\notin D_i$.

\begin{figure}[htbp]
	\centering
	\includegraphics[width=.9\textwidth]{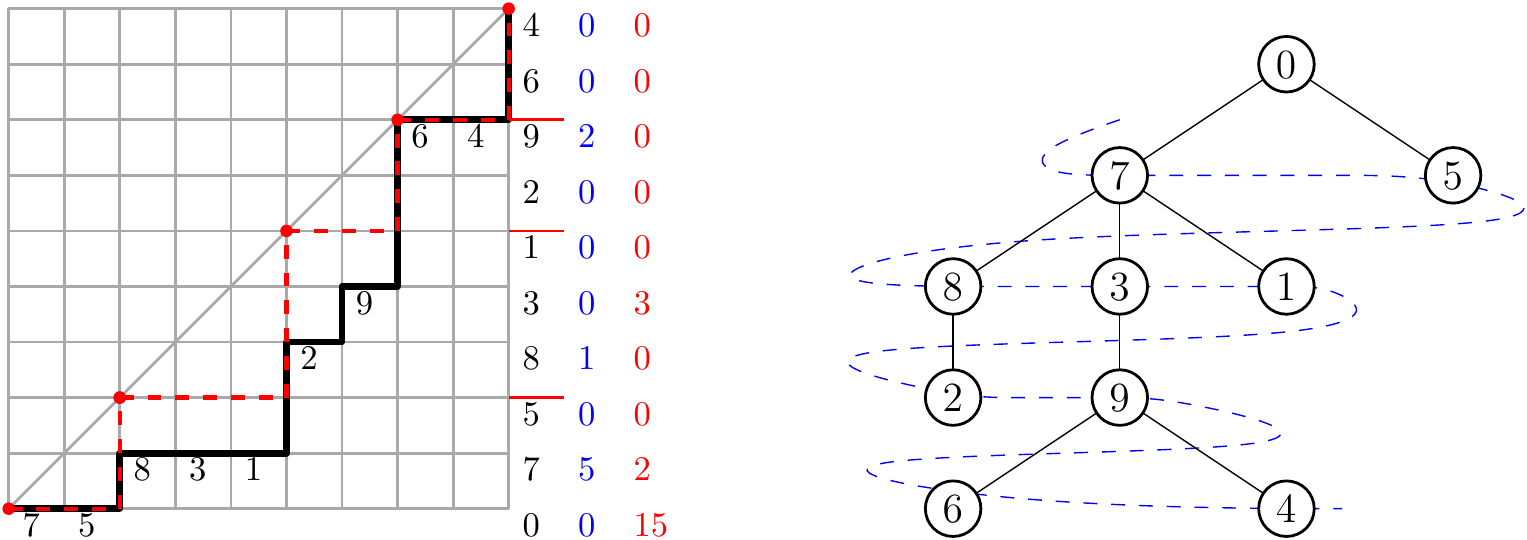}
\caption{Left:  The labelled Dyck path $P$ associated to the parking
	function $p=(1,3,1,7,0,7,0,1,4)$ is displayed, along with its
	bounce path $B$.  Here, the first column (black) from bottom to top is $w=w_0,w_1, \dots,
	w_n$, and the elements of $w$ in between red lines indicate they
	belong to the vertical run of $B$ to their left.  The second (blue) and third (red) columns give $\binv(i)$ and $\bninv(i)$
for the $i$ in the first column.  Right:  The image of $p$ under $\theta$.}
	\label{fig:bounce}
\end{figure}

We define two additional statistics on $p$.  For each $i \in [n]$,
let $\binv(i)$ and $\bninv(i)$ be the number of elements in $D_i$ less than $i$
and greater than $i$, respectively, and set
\begin{equation*}\label{eq:defpinv}
	\pinv(p) = \sum_{i=0}^n \binv(i) \,\,\,\,\, \textnormal{ and }\,\,\,\,\,
	\pninv(p) = \sum_{i=0}^n \bninv(i).
\end{equation*}
Figure \ref{fig:bounce} (left) illustrates the statistics $\binv$ and
$\bninv$.

With this setup, we define a function $\theta : \p_n \rightarrow \T_n$
by defining $C_i$ to be the children of vertex $i$.   See Figure
\ref{fig:bounce} (right).
It is easy to see that $\theta$ is a bijection.  Furthermore, if $p \in \p_n$ and
$T = \theta(p)$, we have:
\begin{itemize}
	\item By definition, for each $i \in [n]$ we have $\binv(i)+ \bninv(i) = |D_i|$.
	\item From \eqref{eq:bouncedi}, it follows that 
		\begin{equation}\label{eq:pinvbounce}
		\pinv(p) + \pninv(p) = \bnce(p).
		\end{equation}
	\item For each $i \in [n]$, the descendants of $i$ in
		$T$ are $D_i$;
	\item For each $i \in [n]$, the statistics $\binv(i)$ and
		$\bninv(i)$ count the number of pairs $(i,j)$ that are
		inversions and coinversions of $T$, respectively.  
	\item It then follows by definition that $\pinv(p)$ and $\pninv(p)$ are equal to
		$\inv(T)$ and $\ninv(T)$, respectively.
\end{itemize}

In light of the above, define for $n \geq 1$
\begin{equation*}
	B_n(q,t) = \sum_{p \in \p_n} q^{\pinv(p)} t^{\pninv(p)},
\end{equation*}
with $B_0(q,t) = 1$;  hence, $B_n(q,q) = \sum_{p \in \p_n} q^{\bnce(p)}$ from \eqref{eq:pinvbounce}.  Furthermore, we see from the discussion
concerning the map $\theta$ that $B_n(q,t) = I_n(q,t)$ for all $n \geq
0$.    The next theorem, which is a refinement of Theorem
\ref{thm:bounce}, now follows from Theorem \ref{thm:mainresult}.
\begin{theorem}
	For all $n \geq 0$, the polynomials $I_n(q, t), F_n(q,t)$ and $B_n(q, t)$ are all
	equal.
	\label{thm:bounceref}
\end{theorem}

\bibliographystyle{amsalpha}
\bibliography{parking_factorizations}
\end{document}